\documentclass[11pt,a4paper,reqno]{amsart}
\usepackage{multirow}
\usepackage{cite}
\usepackage[centertags]{amsmath}
\usepackage{amsfonts}
\usepackage{amssymb}
\usepackage{amsthm}
\usepackage{newlfont}
\usepackage{a4}
\usepackage{enumerate}
\usepackage{graphicx,color}
\theoremstyle{definition}
\newtheorem{defn}{Definition}[section]
\newtheorem{thm}[defn]{Theorem}

\newtheorem{tvr}[defn]{Proposition}

\theoremstyle{remark}
\newtheorem{example}{Example}[section]

\usepackage{bbm}

\newlength{\defbaselineskip}
\newcommand{\setlinespacing}[1]%
           {\setlength{\baselineskip}{#1 \defbaselineskip}}

\newcommand{\map}{\rightarrow}

\newcommand{\q}{\quad}

\renewcommand{\epsilon}{\varepsilon}

\newcommand{\ep}{\varepsilon}
\newcommand{\la}{\lambda}
\newcommand{\al}{\alpha}
\newcommand{\om}{\omega}
\renewcommand{\rho}{\varrho}
\renewcommand{\phi}{\varphi}

\newcommand{\R}{{\mathbb{R}}}

\newcommand{\N}{{\mathbb N}}

\newcommand{\Z}{\mathbb{Z}}
\newcommand{\C}{\mathbb{C}}

\newcommand{\set}[2]{\left\{#1 \, |\, #2 \right\}}
\newcommand{\setb}[2]{\left\{#1 \, \mid\, #2 \right\}}

\newcommand{\abs}[1]{\left\vert#1\right\vert}
\newcommand{\wt}{\widetilde}

\newcommand{\sca}[2]{\langle #1,\, #2\rangle}
\newcommand{\comb}[2]{\begin{pmatrix}
     #1\\
     #2
  \end{pmatrix}}

\begin{document}

\title[Cubature formulas from symmetric orbit functions]
{Cubature formulas of multivariate polynomials arising from symmetric orbit functions}

\author[J. Hrivn\'{a}k]{Ji\v{r}\'{i} Hrivn\'{a}k$^{1}$}
\author[L. Motlochov\'{a}]{Lenka Motlochov\'{a}$^{1,3}$}
\author[J. Patera]{Ji\v{r}\'{i} Patera$^{2,3,4}$}

\date{\today}
\begin{abstract}\small
The paper develops applications of symmetric orbit functions, known from irreducible representations of simple Lie groups, in numerical analysis. It is shown that these functions have remarkable properties which yield to cubature formulas, approximating a weighted integral of any function by a weighted finite sum of function values, in connection with any simple Lie group. The cubature formulas are specialized for simple Lie groups of rank two. An optimal approximation of any function by multivariate polynomials arising from symmetric orbit functions is discussed.
\end{abstract}

\maketitle
\noindent
$^1$ Department of Physics, Faculty of Nuclear Sciences and Physical Engineering, Czech Technical University in Prague, B\v{r}ehov\'a~7, CZ-115 19 Prague, Czech Republic\\
$^2$ Centre de recherches math\'ematiques, Universit\'e de Montr\'eal, C.~P.~6128 -- Centre ville, Montr\'eal, H3C\,3J7, Qu\'ebec, Canada\\
$^3$ D\'epartement de math\'ematiques et de statistique, Universit\'e de Montr\'eal, Qu\'ebec, Canada\\
$^4$ MIND Research Institute, 111~Academy Drive, Irvine, CA 92617
\vspace{10pt}

\noindent
E-mail: jiri.hrivnak@fjfi.cvut.cz, lenka.motlochova@fjfi.cvut.cz, patera@crm.umontreal.ca

\bigskip


\smallskip


\section{Introduction}

The purpose of this paper is to extend the results of \cite{MP4,SsSlcub}, where cubature formulas for numerical integration connected with three types of multivariate Chebyshev-like polynomials arising from Weyl group orbit functions are developed. The specific goal of this article is to derive the cubature rules and corresponding approximation methods for the family of the polynomials arising from symmetric exponential Weyl group orbit sums \cite{Bour,KP1} and detail specializations of the general results for the two-variable polynomials.

The family of polynomials induced by the symmetric Weyl group orbit functions ($C-$functions) forms one of the most natural generalizations of the classical Chebyshev polynomials of one variable --- indeed, the lowest symmetric orbit function arising from the Weyl group of $A_1$ coincides with the common cosine function of one variable and thus induces the family of Chebyshev polynomials of the first kind \cite{Riv}. The continuous and discrete orthogonality of the sets of cosine functions $\cos (nx)$ generalize to the families of multivariate $C-$functions \cite{HP,KP1,MP2}.  This generalization serves as an essential starting point for deriving the cubature formulas and approximation methods.

Cubature formulas for numerical integration constitute multivariate generalizations of classical quadrature formulas for functions of one variable. A weighted integral over some domain inside $\R^n$ of any given function is estimated by a finite weighted sum of values of the same function on a specific set of points (nodes). A standard requirement is imposed: the cubature formula has to hold as an exact equality for polynomials up to a certain degree. Numerous types of cubature formulas with diverse shapes of the integration domains and various efficiencies exist \cite{Cools}. The efficiency of a given cubature formula reflects how the achieved maximal degree of the polynomials relates to the number of the necessary nodes. Optimal cubature formulas of the highest possible efficiency (Gaussian formulas) are for multivariate functions obtained for instance in \cite{LX,MP4,XuC2}. 

The sequence of Gaussian cubature formulas derived in \cite{LX} arises from  the antisymmetric orbit functions ($S-$functions) of the Weyl groups of type $A_n$, $n\in \N$. Generalization of these cubature formulas from \cite{LX} to polynomials of the $S-$functions of Weyl groups of any type and rank is achieved in \cite{MP4}. A crucial concept, which allows the generalization of the $A_n$ formulas, is a novel definition of a degree of the underlying polynomials. This generalized degree ($m-$degree) is based on invariants of the Weyl groups and their corresponding root systems. 
Besides the polynomials corresponding to the $C-$ and $S-$functions, two additional families of multivariate polynomials arise from Weyl group orbit functions of mixed symmetries \cite{SsSlcub}. These hybrid orbit functions ($S^s-$ and $S^l-$functions) exist only for root systems of Weyl groups with two different lengths of roots --- $B_n$, $C_n$, $F_4$ and $G_2$. The cubature formulas related to the polynomials of these $S^s-$ and $S^l-$functions are developed in \cite{SsSlcub}. Deduction of the remaining cubature formulas, which correspond to the polynomials of the $C-$functions, completes in this paper the results of \cite{LX,MP4, SsSlcub}. The integration domains and nodes of these cubature formulas are constructed in a similar way as one-dimensional Gauss-Chebyshev formulas and their Chebyshev nodes.    

Instead of a classical one-dimensional interval, the multivariate $C-$functions are considered in the fundamental domain of the affine Weyl group --- a simplex $F\subset \R^n$. The discrete orthogonality relations of $C-$functions are performed over a finite fragment of a grid $F_M\subset F$, with the parameter $M\in \N$ controlling the density of $F_M$ inside $F$. The simplex $F$ together with the set of points $F_M$ have to be transformed via a transform which induces the corresponding family of polynomials ($X-$transform). This process results in the integration domain $\Omega$ of non-standard shape and the set of nodes $\Omega_M$, with specifically distributed points inside $\Omega$. The last ingredient, needed for successful practical implementation, is the explicit form of the weight polynomial $K$. For practical purposes, the explicit construction of all two-variable cases is presented.

Except for direct numerical integration, one of the most immediate applications of the developed cubature formulas is related multivariate polynomial approximation \cite{CMV2}. The Hilbert basis of the orthogonal multivariate polynomials induced by the $C-$ functions guarantees that any function from the corresponding Hilbert space is expressed as a series involving these polynomials. A specific truncated sum of this expansion provides the best approximation of the function by the polynomials. Among other potential applications of the developed cubature formulas are calculations in  fluid flows \cite{CAB}, laser optics \cite{CTB}, stochastic dynamics \cite{XCL}, magnetostatic modeling \cite{YGA}, micromagnetic simulations \cite{CF}, electromagnetic wave propagation \cite{SMCSHMM}, liquid crystal colloids \cite{TST} and quantum dynamics \cite{LN}. 

The paper is organized as follows. In Section \ref{secroot}, notation and pertinent properties of Weyl groups, affine Weyl groups and $C-$functions are reviewed. In Section \ref{seccub}, the cubature formulas related to $C-$functions are deduced.  In Section \ref{examplecub}, the explicit cubature formulas of the rank two cases $A_2$, $C_2$, $G_2$ are constructed. In Section \ref{secpol}, polynomial approximation methods are developed.

\section{Root systems and polynomials}\label{secroot}
\subsection{Pertinent properties of root systems and weight lattices}\

The notation, established in \cite{HP}, is used. Recall that, to the Lie algebra of the compact, connected, simply connected simple Lie group $G$ of rank $n$, corresponds the set of simple roots $\Delta=(\al_1,\dots,\al_n)$ \cite{Bour, BB,H2,VO}. The set $\Delta$ spans the Euclidean space $\R^n$, with the scalar product denoted by $\sca{\,}{\,}$.
The following standard objects related to the set of simple roots $\Delta$ are used:
\begin{itemize}
\item
The marks $m_1,\dots,m_n$ of the highest root $\xi\equiv -\al_0=m_1\al_1+\dots+m_n\al_n$.

\item
The Coxeter number $m=1+m_1+\dots+m_n$ of $G$.

\item
The Cartan matrix $C$ and its determinant
\begin{equation}\label{Center}
 c=\det C.
\end{equation}

\item
The root lattice $Q=\Z\al_1+\dots+\Z\al_n $.

\item
The $\Z$-dual lattice to $Q$,
\begin{equation*}
 P^{\vee}=\set{\om^{\vee}\in \R^n}{\sca{\om^{\vee}}{\al}\in\Z,\, \forall \al \in \Delta}=\Z \om_1^{\vee}+\dots +\Z \om_n^{\vee}
\end{equation*}
 with the vectors $ \om^\vee_i$ given by $$ \sca{\om^\vee_i}{\al_j}=\delta_{ij}.$$

\item
The dual root lattice $Q^{\vee}=\Z \al_1^{\vee}+\dots +\Z \al^{\vee}_n$, where $\al^{\vee}_i=2\al_i/\sca{\al_i}{\al_i}$.

\item
The dual marks $m^{\vee}_1, \dots ,m^{\vee}_n$ of the highest dual root $\eta\equiv -\al_0^{\vee}= m_1^{\vee}\al_1^{\vee} + \dots + m_n^{\vee} \al_n^{\vee}$. The marks and the dual marks are summarized in Table 1 in \cite{HP}.
The highest dual root $\eta$ satisfies for all $i=1,\dots,n$
\begin{equation}\label{etasca}
 \sca{\eta}{\al_i}\geq 0.
\end{equation}
\item The $\Z$-dual weight lattice to $Q^\vee$
\begin{equation*}
 P=\set{\om\in \R^n}{\sca{\om}{\al^{\vee}}\in\Z,\, \forall \al^{\vee} \in Q^\vee}=\Z \om_1+\dots +\Z \om_n,
\end{equation*}
with the vectors $ \om_i$ given by $ \sca{\om_i}{\al^\vee_j}=\delta_{ij}.$
For $\la\in P$ the following notation is used, \begin{equation}\label{la}\la=\la_1\om_1+\dots + \la_n\om_n= (\la_1,\dots,\la_n). \end{equation} 
\item The partial ordering on $P$ is given: for $\la,\nu \in P$ it holds that  $\nu\leq \la$ if and only if $\la-\nu = k_1\al_1+\dots+ k_n \al_n$ with $k_i \in \Z^{\geq 0}$ for all $i\in \{1,\dots,n\}$. 
\item
The half of the sum of the positive roots
$$\rho =  \om_1+\dots + \om_n. $$
\item The cone of positive weights $P^+$ and the cone of strictly positive weights $P^{++}=\rho +P^+ $ 
\begin{equation*}
P^{+}=\Z^{\geq 0}\om_1+\dots +\Z^{\geq 0} \om_n,\q P^{++}=\N\om_1+\dots +\N \om_n.
\end{equation*}
\item $n$ reflections $r_\al$, $\al\in\Delta$ in $(n-1)$-dimensional `mirrors' orthogonal to simple roots intersecting at the origin denoted by $$r_1\equiv r_{\al_1}, \, \dots, r_n\equiv r_{\al_n}.$$
\end{itemize}

Following \cite{MP4}, we define so called $m-$degree of $\la\in P^+$ as the scalar product of $\la$ with the highest dual root $\eta$, i.e. by the relation 
\begin{equation*}
|\la|_m= \sca{\la}{\eta}	= \la_1 m^\vee_1+\dots +\la_n m^\vee_n.
\end{equation*}
Let us denote a finite subset of the cone of the positive weights $P^+$ consisting of the weights of the  $m-$degree not exceeding $M$ by $P^+_M$, i.e.
\begin{equation*}
P_M^+=\set{\la \in P^+}{\abs{\la}_m \leq M }.	
\end{equation*}
Recall also the separation lemma which asserts for $\la\in P^+$, $\la\neq 0$ and any $M\in \N$ that 
\begin{equation}\label{sep}
|\la|_m < 2M	\q \Rightarrow \q \la\notin MQ.
\end{equation}
Note that this lemma is proved in  \cite{MP4} for $M>m$ only --- the proof, however, can be repeated verbatim with any $M\in \N$. 

For two dominant weights $\la,\,\nu\in P^+$ for which $\nu\leq \la$ we have for their $m-$degrees
\begin{equation}
|\la|_m-	|\nu|_m =  \sca{\la-\nu}{\eta} = \sum_{i=1}^n k_i \sca{\al_i}{\eta}, \q k_i\geq 0.
\end{equation}
Taking into account equation \eqref{etasca}, we have the following proposition.
\begin{tvr}\label{lanu}
For two dominant weights $\la,\,\nu\in P^+$ with $\nu\leq \la$ it holds that $	|\nu|_m \leq|\la|_m$.
\end{tvr}

\subsection{Affine Weyl groups}\

The Weyl group $W$ is generated by $n$ reflections $r_1, \dots, r_n$ and its order $|W|$ can be calculated using the formula
\begin{equation}\label{Weyl}
|W|=n!\, m_1\dots m_n \,c.	
\end{equation}
The affine Weyl group $W^{\mathrm{aff}}$ is the semidirect product of the Abelian group of translations $Q^\vee$ and of the Weyl group~$W$,
\begin{equation}\label{direct}
 W^{\mathrm{aff}}= Q^\vee \rtimes W.
\end{equation}
The fundamental domain $F$ of $W^{\mathrm{aff}}$, which consists of precisely one point of each $W^{\mathrm{aff}}$-orbit, is the convex hull of the points $\left\{ 0, \frac{\om^{\vee}_1}{m_1},\dots,\frac{\om^{\vee}_n}{m_n} \right\}$. Considering $n+1$ real parameters $y_0,\dots, y_n\geq 0$, we have
\begin{align}
F &=\setb{y_1\om^{\vee}_1+\dots+y_n\om^{\vee}_n}{y_0+y_1 m_1+\dots+y_n m_n=1  }. \label{deffun}
\end{align}
The volumes vol$(F)\equiv |F|$ of the simplices $F$ are calculated in \cite{HP}.

Considering the standard action of $W$ on $\R^n$, we denote for $\la \in \R^n$ the isotropy group and its order by $$\mathrm{Stab}(\la)= \set{w\in W}{w\la=\la},\q h_\la\equiv|\mathrm{Stab}(\la)|, $$ and denote the orbit by $$W\la=\set{w\la\in \R^n }{w\in W}.$$ Then the orbit-stabilizer theorem gives for the orders
 \begin{equation}\label{Stab}
|W\la|=\frac{|W|}{h_\la}.
\end{equation} 
Considering the standard action of $W$ on the torus $\R^n/Q^{\vee}$, we denote for $x\in \R^n/Q^{\vee}$ the order of its orbit by $\ep(x)$, i.e.
 \begin{equation}\label{ep}
\ep(x)=\abs{\set{wx\in \R^n/Q^{\vee} }{w\in W}}.
\end{equation}

For an arbitrary $M\in\N$, the grid $F_M$ is given as cosets from the $W-$invariant group $\frac{1}{M}P^{\vee}/Q^{\vee}$ with a representative element in the fundamental domain $F$
\begin{equation*}
F_M\equiv\frac{1}{M}P^{\vee}/Q^{\vee}\cap F.
 \end{equation*}
The representative points of $F_M$ can be explicitly written as
\begin{equation}\label{FM}
 F_M = \setb{\frac{u_1}{M}\om^{\vee}_1+\dots+\frac{u_n}{M}\om^{\vee}_n}{u_0,u_1,\dots ,u_n \in \Z^{\geq 0},\, u_0+u_1m_1+\dots + u_n m_n=M}.
 \end{equation}
The numbers of elements of $F_M$, denoted by $|F_M|$, are also calculated in \cite{HP} for all simple Lie algebras.

\subsection{Orbit functions}\

Symmetric orbit functions \cite{KP2} are defined as complex functions  $C_\la:\R^n\map \C$ with the labels $\la\in P^{+}$,
\begin{equation}\label{Cgenorb}
C_\la(x)=\sum_{\nu\in W\la}e^{2 \pi i \sca{ \nu }{x}},\q x\in \R^n.
\end{equation}
Note that in \cite{HP} the results for $C-$functions are formulated for the normalized $C-$functions $\Phi_\la$ which are related to the orbit sums \eqref{Cgenorb} as  $$\Phi_\la = h_\la \,C_\la .$$
Due to the symmetries with respect to the Weyl group $W$ as well as with respect to the shifts from $Q^\vee$
\begin{equation}\label{winv}
C_\la(wx)=C_\la(x),\q C_\la(x+q^\vee)=C_\la(x),\q w\in W,\, q^\vee\in Q^\vee, \end{equation}
it is sufficient to consider $C-$functions restricted to the fundamental domain of the affine Weyl group~$F$. Moreover, the $C-$functions are continuously orthogonal on $F$,
\begin{equation}\label{intor}
\int_F C_\la(x)\overline{C_{\la'}(x)}\,dx=\, \frac{|F||W|}{h_\la}\, \delta_{\la,\la'}\,.
\end{equation}
and form a Hilbert basis of the space $\mathcal{L}^2(F) $ \cite{KP2}, i.e. any function $\wt  f\in\mathcal{L}^2(F)$ can be expanded into the series of $C-$functions
\begin{equation}\label{expansion}
\wt f=\sum_{\la\in P^+}c_\la C_\la, \q  	c_\la= \frac{h_\la}{|F| |W|}\int_F \wt f(x)\overline{C_{\la}(x)}\,dx .
\end{equation}
Special case of the orthogonality relations \eqref{intor} is when one of the weights is equal to zero,
\begin{equation}\label{intort}\int_F C_\lambda(x)\,dx=\,|F|\, \delta_{\lambda,0}\,.\end{equation}
For any $M\in N$, the $C-$functions from a certain subset of $P^+$  are also discretely orthogonal on $F_M$ and form a basis of the space of discretized functions $\C^{F_M}$ of dimension $|F_M|$ \cite{HP}; special case of these orthogonality relations is when one of the weights is equal to zero modulo the lattice $MQ$,
\begin{equation}\label{disort}\sum_{x\in F_M}\ep(x) C_\lambda(x)=\begin{cases}
cM^n&\lambda\in MQ,\\ 0&\lambda\notin MQ.\end{cases}\end{equation}
The key point in developing the cubature formulas is comparison of  formulas \eqref{intort} and \eqref{disort} in the following proposition. 
\begin{tvr}\label{con}
For any $M\in \N$ and $\la\in P^+_{ 2M-1}$ it holds that 
\begin{equation}
\frac{1}{|F|}	\int_F C_\lambda(x)\,dx =\frac{1}{cM^n} \sum_{x\in F_M}\ep(x) C_\lambda(x).
\end{equation}
\end{tvr}
\begin{proof}
Suppose first that $\la=0$. Then from \eqref{intort} and \eqref{disort} we obtain
$$\frac{1}{|F|}	\int_F C_0(x)\,dx =1 =  \frac{1}{cM^n} \sum_{x\in F_M}\ep(x) C_0(x)  $$
 Secondly let $\la\neq 0$ and $|\la|_m < 2M$. The from the separation lemma \eqref{sep} we have that $\la\notin MQ$ and thus 
$$\frac{1}{|F|}	\int_F C_\la(x)\,dx =0 =  \frac{1}{cM^n} \sum_{x\in F_M}\ep(x) C_\la(x).  $$
\end{proof}

Let us denote for convenience the $C-$functions corresponding to the basic dominant weights $\om_j $ by  $Z_j$, i.e.
\begin{equation*}
Z_j \equiv C_{\omega_j}.	
\end{equation*}

Recall from \cite{Bour}, Ch. VI, \S 4 that any $W-$invariant sum of the exponential functions  $e^{2 \pi i \sca{ \nu }{a}}$ can be expresssed as a linear combination of some functions $C_\la $ with  $\la\in P^+ $. Also for any $\la\in P^+ $ a function of the monomial type $Z_1^{\la_1}Z_2^{\la_2}\dots Z_n^{\la_n}$ can be expressed as the sum of  $C-$functions by less or equal dominant weights than $\la$, i.e.
\begin{equation}\label{mon}
 Z_1^{\la_1}Z_2^{\la_2}\dots Z_n^{\la_n}= \sum_{\nu\leq\la,\,\nu\in P^+}c_{\nu}C_{\nu},\q c_\nu \in \C, \q c_\la =1.
\end{equation}
Conversely, any function $C_\la$, $\la\in P^+$ can be expressed as a polynomial in variables $Z_1,\dots,Z_n$, i.e. there exist a multivariate polynomials $\wt p_\la \in \C [y_1,\dots,y_n]$ such that
\begin{equation}\label{Cpoly}
C_\la = \wt	p_\la (Z_1,\dots, Z_n)=  \sum_{\nu\leq\la,\,\nu\in P^+}d_{\nu}Z_1^{\nu_1}Z_2^{\nu_2}\dots Z_n^{\nu_n},\q d_\nu \in \C,\q d_\la =1.
\end{equation}

Antisymmetric orbit functions \cite{KP2} are defined as complex functions  $S_\la:\R^n\map \C$ with the labels $\la\in P^{++}$,
\begin{equation}\label{Sgenorb}
S_\la(x)=\sum_{w\in W}\det (w)\, e^{2 \pi i \sca{ w\la}{x}},\q x\in \R^n.
\end{equation}
The antisymmetry with respect to the Weyl group $W$ and the symmetry with respect to the shifts from $Q^\vee$ holds
\begin{equation*}
S_\lambda(wx)=(\det w )\, S_\lambda(x),\q S_\lambda(x+q^\vee)=S_\lambda(x),\q w\in W,\, q^\vee\in Q^\vee, \end{equation*}
Recall that Proposition 9 in \cite{KP2} states that for the lowest $S-$function $S_\rho$ it holds that
\begin{align}
S_\rho (x)&=0, \q a\in F\setminus F^\circ\\
S_\rho(x)&\neq 0,   \q a\in  F^\circ. \label{nonzero}
\end{align}
Since the square of the absolute value $| S_\rho |^2=S_\rho \overline{S_\rho} $ is a $W-$invariant sum of exponentials it can be expressed as a linear combinations of  $C-$functions. Each $C-$function in this combination is moreover a polynomial of the form \eqref{Cpoly}.  Thus there exist a unique polynomial  
$\wt K \in \C [y_1,\dots,y_n]$ such that
\begin{equation}\label{Kt}
| S_\rho |^2= \wt K (Z_1,\dots, Z_n).
\end{equation}

\section{Cubature formulas}\label{seccub}
\subsection{The $X-$transform}\

The key component in the development of the cubature formulas is the integration by substitution. The $X-$transform transforms the  fundamental $F\subset \R^n$ domain to the domain $\Omega\subset \R^n$ on which are the cubature rules defined.  In order to obtain a real valued transform we first need to examine the values of the $C-$functions.  

The $C-$functions of the algebras
\begin{equation}\label{realval}
	A_1,B_n(n\geq3),C_n(n\geq2),D_{2k}(k\geq2),E_7,E_8,F_4,G_2
\end{equation}
are real-valued \cite{KP1}. Using the notation \eqref{la}, for the remaining cases it holds that
\begin{align}\label{complex}
A_n(n\geq2):&\quad C_{(\lambda_1,\lambda_2,\dots,\lambda_n)}(x)=\overline{C_{(\lambda_n,\lambda_{n-1},\dots,\lambda_1)}(x)}\,,\nonumber\\ 
D_{2k+1}(k\geq2):&\quad C_{(\lambda_1,\lambda_2,\dots,\lambda_{2k-1},\lambda_{2k},\lambda_{2k+1})}(x)=\overline{C_{(\lambda_1,\lambda_2,\dots,\lambda_{2k-1},\lambda_{2k+1},\lambda_{2k})}(x)}\,,\\
E_6:&\quad C_{(\lambda_1,\lambda_2,\lambda_{3},\lambda_{4},\lambda_{5},\lambda_6)}(x)=\overline{C_{(\lambda_5,\lambda_4,\lambda_{3},\lambda_{2},\lambda_{1},\lambda_6)}(x)}\,. \nonumber
\end{align}
Specializing the relations \eqref{complex} for  the $C-$functions corresponding to the basic dominant weights $Z_j$, we obtain that the functions $Z_j$ are real valued, except for the following cases for which it holds that 
\begin{align}\label{real}
A_{2k} (k\geq1)\,:&\quad Z_j=\overline{Z_{2k-j+1}}, \,j=1,\dots,k\nonumber\\ 
A_{2k+1} (k\geq1)\,:&\q Z_j=\overline{Z_{2k-j+2}}, \,j=1,\dots,k\nonumber\\ 
D_{2k+1}(k\geq2):&\quad  Z_{2k}=\overline{Z_{2k+1}},\\
E_6:&\quad Z_2=\overline{Z_4},Z_1=\overline{Z_5}. \nonumber
\end{align}
Taking into account \eqref{real}, we introduce the real-valued functions $X_j,\, j\in\{1,\dots,n\}$ as follows.  For the cases \eqref{realval} we set
\begin{equation}\label{var1}
X_j\equiv Z_j,
\end{equation}
and for the remaining cases \eqref{real} we define
\begin{equation}\label{var2}
\begin{alignedat}{4}
&A_{2k}:&\quad &X_j=\frac{Z_j+Z_{2k-j+1}}{2}\,,X_{2k-j+1}=\frac{Z_j-Z_{2k-j+1}}{2i},\, j=1,\dots, k\,;\\
&A_{2k+1}:&\quad &X_j=\frac{Z_j+Z_{2k-j+2}}{2}\,,X_{k+1}=Z_{k+1}\,,X_{2k-j+2}=\frac{Z_j-Z_{2k-j+2}}{2i},\, j=1,\dots, k,\\
&D_{2k+1}:&\quad& X_j=Z_j\,,X_{2k}=\frac{Z_{2k}+Z_{2k+1}}{2}\,,X_{2k+1}=\frac{Z_{2k}-Z_{2k+1}}{2i}\,,j=1,\dots,2k-1\,;\\
&E_6:&\quad& X_1=\frac{Z_1+Z_5}{2}\,,X_2=\frac{Z_2+Z_4}{2}\,,X_3=Z_3\,,X_4=\frac{Z_2-Z_4}{2i}\,,X_5=\frac{Z_1-Z_5}{2i},X_6=Z_6.
\end{alignedat}
\end{equation}
Thus, we obtain a crucial mapping $X:\R^n\map\R^n$ given by
\begin{equation}\label{Xtrans}
X(x)\equiv(X_1(x),\dots,X_n(x)).	
\end{equation}

The image $\Omega\subset \R^n$ of  the fundamental domain $F$ under the mapping $X$ forms the integration domain on which the cubature rules will be formulated, i.e.
\begin{equation}\label{Omega}
\Omega\equiv X(F).
\end{equation}
In order to use the mapping $X$ for an integration by substitution we need to know that it is one-to-one except for possibly some set of zero measure. 
Since the image $\Omega_M\subset \R^n$ of  the set of  points $F_M$ under the mapping $X$ forms the set of nodes for the cubature rules, i.e.
\begin{equation}\label{OmegaM}
\Omega_M\equiv X(F_M),
\end{equation}
a discretized version of the one-to-one correspondence of the restriced mapping $X_M$ of $X$ to $F_M$, i.e.
\begin{equation}\label{Xtransd}
X_M\equiv X\restriction_{F_M}
\end{equation}
is also essential. Note that due to the periodicity of $C-$functions \eqref{winv}, the restriction \eqref{Xtransd} is well-defined for the cosets from $F_M$. 
\begin{tvr}\label{corr}
The mapping $X:F\map\Omega$, given by \eqref{Xtrans}, is one-to-one correspondence except for some set of zero measure. For any $M\in \N$ is the restriction mapping $X_M:F_M\map\Omega_M$, given by \eqref{Xtransd}, one-to-one correspondence and thus it holds that 
\begin{equation}
|\Omega_M|= |F_M|.	
\end{equation}
\end{tvr}	
\begin{proof}
Let us assume that there exists a set $F'\subset F$ of non-zero measure such that $X(x)=X(y)$ with $x,y\in F'$. Since the transforms  \eqref{var1}, \eqref{var2} are as regular linear mappings one-to-one correspondences, this fact implies that $Z_1(x)=Z_1(y),\dots, Z_n(x)=Z_n(y) $ with $x,y\in F'$. Then from the polynomial expression \eqref{Cpoly} we obtain for all $\la\in P^+$ that it holds that $C_\la(x)=C_\la(y)$. Since the $C-$functions $C_\la,\,\la\in P^+  $ form a Hilbert basis of the space $\mathcal{L}^2(F) $ we conclude that for any $f\in\mathcal{L}^2(F) $ is valid that $f(x)=f(y)$, $x,y\in F'$ which is contradiction.

Retracing the steps of the continuous case above, let us assume that there exist two distinct points  $x,y\in F_M$, $x\neq y $ such that $X(x)=X(y)$. Since the transforms  \eqref{var1}, \eqref{var2} are as regular linear mappings one-to-one correspondences, this fact again implies that $Z_1(x)=Z_1(y),\dots, Z_n(x)=Z_n(y) $. Then from the polynomial expression \eqref{Cpoly} we obtain for all $\la\in P^+$ that it holds that $C_\la(x)=C_\la(y)$. The same equality has hold for those $C-$functions $C_\la$ which form a basis of the space $\C^{F_M}$. We conclude that for any $f\in\C^{F_M}$ is valid that $f(x)=f(y)$, $x\neq y$ which is contradiction.
\end{proof}

The absolute value of the determinant of the Jacobian matrix of the $X$ transform \eqref{Xtrans} is essential for construction of the cubature formulas ---   its value is determined in the following proposition.
\begin{tvr}\label{jac}
The absolute value of the Jacobian determinant $|J_x(X)|$ of the $X-$transform \eqref{Xtrans} is given by
\begin{equation}
|J_x(X)| = \frac{\kappa (2\pi)^n }{|F||W|}	|S_\rho (x)|,
\end{equation}
where $\kappa$ is defined as
\begin{equation}
\kappa=\begin{cases}
2^{-\lfloor\frac{n}{2}\rfloor}& \text{for } A_n\\
\frac{1}{2}&\text{for } D_{2k+1}\\
\frac{1}{4}&\text{for }  E_6\\
1&\text{otherwise.}
\end{cases}
\end{equation}
\end{tvr}
\begin{proof}
Note that the $X$ transform can be composed of the the following two
transforms: the transform $\zeta: x\mapsto (Z_1(x),\dots,Z_n(x))   $ and the transform $R: (Z_1,\dots,Z_n)\mapsto (X_1,\dots,X_n)$ via relations  \eqref{var1}, \eqref{var2}. To calculate the Jacobian of the transform $\zeta$, let us denote by $\al^\vee$ the matrix of the coordinates (in columns) of the vectors $\al^\vee_1,\dots, \al^\vee_n$ in the standard orthonormal basis of $\R^n$ and by $a_1,\dots,a_n$ the coordinates of a point $x\in \R ^n$ in $\al^\vee-$basis, i.e. $x=a_1\al^\vee_1+\dots+a_n \al^\vee_n$. 
If $a$ denotes the coordinates $a_1,\dots,a_n$ arranged in a column vector then it holds that $x=\al^\vee a$.
The absolute value of the Jacobian of the mapping $a\mapsto  (Z_1 (\al^\vee a) ,\dots,Z_n(\al^\vee a)) $ is according to equation (32) in \cite{MP4} given by $(2\pi)^n |S_\rho (\al^\vee a)|$. Using the chain rule, this implies  for the absolute value of the Jacobian $|J_x(\zeta)|$ of the map $\zeta$ that
\begin{equation*}
	|J_x(\zeta)|= |\det \al^\vee|^{-1}(2\pi)^n |S_\rho (x)|.
\end{equation*}
It can be seen directly from formula \eqref{Weyl} and Proposition 2.1 in \cite{HP} that 
\begin{equation*}
 |\det \al^\vee|=|W| |F|.
\end{equation*}
The calculation of the absolute value of the Jacobian determinant $\kappa= |J_x(R)|$ is straightforward from definitions  \eqref{var1}, \eqref{var2}.
\end{proof}


\subsection{The cubature formula}\

We attach to any $\la \in P^+$ a monomial $y^\la\equiv y_1^{\la_1}\dots y_n^{\la_n}\in \C [y_1,\dots,y_n]$ and assign to this monomial the  $m-$degree $|\la|_m$ of $\la$.
The $m-$degree of a polynomials $p\in \C [y_1,\dots,y_n]$, denoted by $\mathrm{deg}_m\,p$, is defined as the largest $m-$degree of a monomial occurring in $p(y)\equiv p(y_1,\dots,y_n)$. 
For instance we observe from Proposition \ref{lanu} and \eqref{Cpoly} that the $m-$degree of the $C-$polynomials $\wt p_\la$ coincides with the  $m-$degree of~$\la$,
\begin{equation}\label{degwtp}
	\mathrm{deg}_m\,\wt p_\la =|\la|_m .
\end{equation}
The subspace $\Pi_M\subset  \C [y_1,\dots,y_n]$ is formed by the polynomials of  $m-$degree at most $M$, i.e. 
\begin{equation}\label{mdegree}
\Pi_M\equiv \set{ p\in   \C [y_1,\dots,y_n]}{\mathrm{deg}_m\,p\leq M}\,.
\end{equation}
In order to investigate how the $m-$degree of a polynomial changes under the substitution of the type  \eqref{var2} we formulate the following proposition.
\begin{tvr}\label{subslem}
Let $j,k\in \{1,\dots,n\}$ be two distinct indices $j< k$ such that $m^\vee_j=m^\vee_k$ and $p , \wt p\in \C [y_1,\dots,y_n]$ two polynomials such that
\begin{equation*}
\wt p (y_1,\dots,y_n )= p 	(y_1,\dots,y_{j-1}, \frac{y_j+y_k}{2}  ,\dots,y_{k-1}, \frac{y_j-y_k}{2i},\dots  y_n )
\end{equation*}
holds. Then $\mathrm{deg}_m \wt p=\mathrm{deg}_m p$.
\end{tvr}
\begin{proof}
Since any polynomial $p\in \Pi_M$ is a linear combination of monomials $y^\lambda$ with $\abs{\lambda}_m\leq M$, it is sufficient to prove  $\wt p \in \Pi_M$ for all monomials $p$ of $m-$degree at most $M$.
If $p$ is a monomial $y^\lambda$ with $\abs{\lambda}_m\leq M$, then
$$\wt p(y_1,\dots, y_n) =\left(\frac{y_j+y_k}{2}\right)^{\lambda_j}\left(\frac{y_j-y_k}{2i}\right)^{\lambda_k}\prod_{l\in\{1,\dots,n\}\setminus \{j,k\}}y_l^{\lambda_l}\,.$$
Using the binomial expansion, we obtain
$$\wt p(y_1,\dots, y_n) =\frac{1}{i^{\lambda_k}2^{\lambda_j+\lambda_k}}\sum_{r=0}^{\lambda_j}\sum_{s=0}^{\lambda_k}(-1)^{\lambda_k-s}\comb{\lambda_j}{ r }\comb{\lambda_k}{ s } y_j^{r+s}y_k^{\lambda_j+\lambda_k-(r+s)}\prod_{l\in\{1,\dots,n\}\setminus \{j,k\}}y_l^{\lambda_l}\,.$$
Therefore, the $m-$degree of the polynomial $\wt p$ is given by
$$\mathrm{deg}_m \wt p=\max_{r,s}\{(r+s)m_j^\vee+(\lambda_j+\lambda_k-(r+s))m_k^\vee+ \sum_{l\in\{1,\dots,n\}\setminus \{j,k\}}\lambda_lm_l^\vee\}\,.$$
Since we assume that $m_j^\vee=m_k^\vee$, we conclude that 
$\mathrm{deg}_m \,\wt p=\sum_{l=1}^n\lambda_lm_l^\vee=\mathrm{deg}_m \, p$.
\end{proof}

Having the  $X_M$ transform \eqref{Xtransd}, it is possible to transfer uniquely the values \eqref{ep} of $\ep(x),\, x\in F_M$ to the points of $\Omega_M$, i.e. by the relation
\begin{equation}
	\wt\ep (y)\equiv \ep (X^{-1}_My),\q y\in \Omega_M.
\end{equation}
Taking the inverse transforms  of \eqref{var1}, \eqref{var2} and substituting them into the polynomials \eqref{Kt}, \eqref{Cpoly} we obtain the polynomials $K,p_\la \in\C [y_1,\dots,y_n]$ such that
\begin{equation}\label{KK}
	| S_\rho |^2 =  K (X_1,\dots, X_n),
\end{equation}

\begin{thm}[Cubature formula] For any $M\in\N$ and any  $p\in \Pi_{2M-1}$ it holds that
\begin{equation}\label{gcub}
\int_{\Omega}p(y)K^{-\frac{1}{2}}(y)\, dy=\frac{\kappa}{c|W|}\left(\frac{2\pi}{M}\right)^n\sum_{y\in \Omega_M}\wt\ep(y)p(y)\,.\end{equation}
\end{thm}
\begin{proof}
Proposition \ref{corr} guarantees that the $X-$transform is one-to-one except for some set of measure zero and Proposition \ref{jac} together with \eqref{nonzero} gives that the Jacobian determinant is non-zero except for the boundary of $F$.  Thus using the integration by substitution $y=X(x)$ we obtain
$$\int_{\Omega}p(y)K^{-\frac{1}{2}}(y)\, dy =\frac{\kappa (2\pi)^n}{|F||W|}\int_{F}p(X(x))\, dx .$$
The one-to-one correspondence for the points $F_M$ and $\Omega_M$ from Proposition \ref{corr} enables us to rewrite the finite sum in \eqref{gcub} as
\begin{equation}\label{undev}
	\frac{\kappa}{c|W|}\left(\frac{2\pi}{M}\right)^n\sum_{y\in \Omega_M}\wt\ep(y)p(y)=\frac{\kappa}{c|W|}\left(\frac{2\pi}{M}\right)^n\sum_{x\in F_M}\ep(x)p(X(x)) .
\end{equation}
Succesively applying Proposition \ref{subslem} to perform the substitutions \eqref{var2} in $p$ we conclude that there exist a polynomial $\wt p\in \Pi_{2M-1 }$ such that $\wt p (Z_1,\dots,Z_n)= p(X_1,\dots,X_n) .$
Due to \eqref{mon} we obtain for the polynomial $\wt p$ that $$p(X_1,\dots,X_n) =\wt p (Z_1,\dots,Z_n)= \sum_{\la \in P^+_{2M-1}}\wt c_\la Z_1^{\la_1}Z_2^{\la_2}\dots Z_n^{\la_n}=\sum_{\la \in P^+_{2M-1}}\wt c_\la \sum_{\nu\leq\la,\,\nu\in P^+}c_{\nu}C_{\nu} $$
and therefore it holds that
\begin{equation}\label{konec1}
	\frac{1}{|F|}\int_{F}p(X(x))\, dx= \sum_{\la \in P^+_{2M-1}}\wt c_\la \sum_{\nu\leq\la,\,\nu\in P^+}c_{\nu}	\frac{1}{|F|}\int_{F}C_{\nu}(x)\, dx
\end{equation}
and 
\begin{equation}\label{konec2}
\frac{1}{cM^n}\sum_{x\in F_M}\ep(x)p(X(x))= \sum_{\la \in P^+_{2M-1}}\wt c_\la \sum_{\nu\leq\la,\,\nu\in P^+}c_{\nu}	\frac{1}{cM^n}\sum_{x\in F_M}\ep(x)C_{\nu}(x).
\end{equation}
Since Proposition \ref{lanu} states that for all $\nu\leq \la$ it holds that $|\nu|_m\leq |\la|_m\leq 2M-1$, we connect equations \eqref{konec1} and \eqref{konec2} by Proposition \ref{con}.
\end{proof}
Note that for practical purposes it may be more convenient to use the cubature formula \eqref{gcub} in its less developed form resulting from \eqref{undev},
\begin{equation}\label{gcub2}
\int_{\Omega}p(y)K^{-\frac{1}{2}}(y)\, dy=\frac{\kappa}{c|W|}\left(\frac{2\pi}{M}\right)^n\sum_{x\in F_M}\ep(x)p(X(x)) \,.\end{equation}
This form may be more practical since  the explicit inverse transform to $X_M$ is usually not available and, on the contrary, the calculation of the coefficients $\ep(x)$ and the points $X(x),\, x\in F_M$ is straightforward.

\section{Cubature formulas of rank two }\label{examplecub}\ 

In this section we specialize the cubature formula for the irreducible root systems of rank two. Let us firstly recall  some basic facts about root systems of rank $2$, i.e., $A_2,C_2$ and $G_2$. They are characterized by two simple roots $\Delta= (\alpha_1,\alpha_2)$ which satisfy
\begin{equation}
\begin{alignedat}{4}
A_2:&\quad&\langle\alpha_1,\alpha_1\rangle=2\,,&\quad&\langle\alpha_2,\alpha_2\rangle=2\,,&\quad&\langle\alpha_1,\alpha_2\rangle=-1\\\
C_2:&\quad&\langle\alpha_1,\alpha_1\rangle=1\,,&\quad&\langle\alpha_2,\alpha_2\rangle=2\,,&\quad&\langle\alpha_1,\alpha_2\rangle=-1\\
G_2:&\quad&\langle\alpha_1,\alpha_1\rangle=2\,,&\quad&\langle\alpha_2,\alpha_2\rangle=\frac23\,,&\quad&\langle\alpha_1,\alpha_2\rangle=-1\,.
\end{alignedat}
\end{equation}  
The transformation rules among the root system and the remaining three bases is given as follows,
\begin{equation}
\begin{alignedat}{9}
A_2:&\quad&\alpha_1&=&\alpha_1^\vee=2\omega_1-\omega_2\,,&
\quad&\alpha_2&=&\alpha^\vee_2=-\omega_1+2\omega_2\,,
&\quad\quad&\omega_1&=&\omega_1^\vee\,,
&\quad&\omega_2&=&\omega_2^\vee\,\\
C_2:&\quad&\alpha_1&=&\frac12\alpha_1^\vee=2\omega_1-\omega_2\,,&
\quad&\alpha_2&=&\alpha^\vee_2=-2\omega_1+2\omega_2\,,
&\quad\quad&\omega_1&=&\frac12\omega^\vee_1\,,
&\quad&\omega_2&=&\omega_2^\vee\,\\
G_2:&\quad&\alpha_1&=&\alpha^\vee_1=2\omega_1-3\omega_2\,,
&\quad&\alpha_2&=&\frac13\alpha_2^\vee=-\omega_1+2\omega_2\,,
&\quad\quad&\omega_1&=&\omega_1^\vee\,,
&\quad&\omega_2&=&\frac13\omega_2^\vee\,.
\end{alignedat}
\end{equation}
Taking the weights in the standard form $\lambda=(\lambda_1,\lambda_2)=\lambda_1\omega_1+\lambda_2\omega_2$, the corresponding Weyl group is generated by reflections $r_1$ and $r_2$ of the explicit form
\begin{equation}
\begin{alignedat}{5}
A_2:&\quad&r_1(\lambda_1,\lambda_2)&=&(-\lambda_1,\lambda_1+\lambda_2)\,,&\quad& r_2(\lambda_1,\lambda_2)&=&(\lambda_1+\lambda_2,-\lambda_2)\,\\
C_2:&\quad&r_1(\lambda_1,\lambda_2)&=&(-\lambda_1,\lambda_1+\lambda_2)\,,&\quad& r_2(\lambda_1,\lambda_2)&=&(\lambda_1+2\lambda_2,-\lambda_2)\,\\
G_2:&\quad&r_1(\lambda_1,\lambda_2)&=&(-\lambda_1,3\lambda_1+\lambda_2)\,,&\quad& r_2(\lambda_1,\lambda_2)&=&(\lambda_1+\lambda_2,-\lambda_2)\,.
\end{alignedat}
\end{equation}
Any Weyl group orbit of a generic point $\lambda\in P$ consists of 
\begin{equation}\label{orbit}
\begin{aligned}
A_2:\quad&\{(\lambda_1,\lambda_2)\,,(-\lambda_1,\lambda_1+\lambda_2)\,,(\lambda_1+\lambda_2,-\lambda_2)\,,(-\lambda_2,-\lambda_1)\,,(-\lambda_1-\lambda_2,\lambda_1)\,,(\lambda_2,-\lambda_1-\lambda_2)\}\,;\\
C_2:\quad&\{\pm(\lambda_1,\lambda_2)\,,\pm(-\lambda_1,\lambda_1+\lambda_2)\,,\pm(\lambda_1+2\lambda_2,-\lambda_2)\,,\pm(\lambda_1+2\lambda_2,-\lambda_1-\lambda_2)\}\,;\\
G_2:\quad&\{\pm(\lambda_1,\lambda_2)\,,\pm(-\lambda_1,3\lambda_1+\lambda_2)\,,\pm(\lambda_1+\lambda_2,-\lambda_2)\,,\pm(2\lambda_1+\lambda_2,-3\lambda_1-\lambda_2)\,,\\ &\pm(-\lambda_1-\lambda_2,3\lambda_1+2\lambda_2)\,,\pm(-2\lambda_1-\lambda_2,3\lambda_1+2\lambda_2)\}\,.
\end{aligned}
\end{equation}

\subsection{The case $A_2$}\

If the points are considered in the $\al^\vee-$basis, $x=a_1\alpha_1^\vee+a_2\alpha_2^\vee$, the symmetric  $C-$functions \eqref{orbit} and antisymmetric $S-$functions of $A_2$ are explicitly given by
\begin{equation*}
\begin{aligned}
C_{(\lambda_1,\lambda_2)}(a_1,a_2)&=\frac{1}{h_\lambda}\left(e^{2\pi i(\lambda_1a_1+\lambda_2a_2)}+e^{2\pi i(-\lambda_1a_1+(\lambda_1+\lambda_2)a_2)}+e^{2\pi i((\lambda_1+\lambda_2)a_1-\lambda_2a_2)}\right.\\ &\left.+e^{2\pi i(\lambda_2a_1-(\lambda_1+\lambda_2)a_2)}+e^{2\pi i((-\lambda_1-\lambda_2)a_1+\lambda_1a_2)}+e^{2\pi i(-\lambda_2a_1-\lambda_1a_2)}\right)\,,\\
S_{(\lambda_1,\lambda_2)}(a_1,a_2)&=e^{2\pi i(\lambda_1a_1+\lambda_2a_2)}-e^{2\pi i(-\lambda_1a_1+(\lambda_1+\lambda_2)a_2)}-e^{2\pi i((\lambda_1+\lambda_2)a_1-\lambda_2a_2)}\\ &+e^{2\pi i(\lambda_2a_1-(\lambda_1+\lambda_2)a_2)}+e^{2\pi i((-\lambda_1-\lambda_2)a_1+\lambda_1a_2)}-e^{2\pi i(-\lambda_2a_1-\lambda_1a_2)}\,,
\end{aligned}
\end{equation*} 
where the values $h_\lambda$ can be found in Table \ref{hl}.
\begin{table}
\begin{tabular}{c|ccc}
&&$h_\lambda$&\\\hline
$\lambda\in P^+$&$A_2$&$C_2$&$G_2$\\
\hline\hline
$(0,0)$&6&8&12\\
$(\star,0)$&2&2&2\\
$(0,\star)$&2&2&2\\
$(\star,\star)$&1&1&1\\
\end{tabular}
\medskip
\caption{The values of $h_\lambda$ are shown for $A_2,C_2$ and $G_2$ with $\star$ denoting the corresponding coordinate different from $0$.}
\label{hl}
\end{table}
Performing the transform \eqref{var2}, the resulting real-valued functions $X_1,\,X_2$ are given by
\begin{equation*}
X_1(a_1,a_2)=\cos{2\pi a_1}+\cos{2\pi a_2}+\cos{2\pi(a_1-a_2)}\,,\quad X_2(a_1,a_2)=\sin{2\pi a_1}-\sin{2\pi a_2}-\sin{2\pi(a_1-a_2)}\,.
\end{equation*}
The integral domain $\Omega$ can be described explicitly as
\begin{equation*}
\Omega=\set{(y_1,y_2)\in\R^2}{-(y_1^2+y_2^2+9)^2+8(y_1^3-3y_1y_2^2)+108\geq0}
\end{equation*}
The weight $K-$polynomial \eqref{KK} is given explicitly as
\begin{equation*}
K(y_1,y_2)=-(y_1^2 + y_2^2 + 9)^2 + 8(y_1^3 - 3 y_1y_2^2) + 108\,.
\end{equation*}
The index set which will label the sets of points $F_M$ and $\Omega_M$ is introduced via
\begin{align*}
I_M=\set{[s_0,s_1,s_2]\in(\Z^{\ge 0})^3}{s_0+ s_1+ s_2=M }.
\end{align*}
Thus the grid $F_M$  consists of the points
\begin{equation*}
F_M=\set{\frac{s_1}{M}\omega_1^\vee+\frac{s_2}{M}\omega_2^\vee}{ [s_0,s_1,s_2]\in I_M}\,.
\end{equation*}
If for $j=[s_0,s_1,s_2]\in I_M$ we denote 
\begin{equation*}
(y_1^{(j)}, 	y_2^{(j)})= \left(X_1\left(\frac{2s_1+s_2}{3M},\frac{s_1+2s_2}{3M}\right),X_2\left(\frac{2s_1+s_2}{3M},\frac{s_1+2s_2}{3M}\right)\right)
\end{equation*}
then the set of nodes $\Omega_M$ consists of the points
\begin{equation*}
\Omega_M=\setb{(y_1^{(j)}, 	y_2^{(j)})\in \R^2}{ j\in I_M}\,.
\end{equation*}
The integration domain $\Omega$ is together with the set of nodes $\Omega_{15}$ depicted in Figure \ref{A2}.
\begin{figure}
\includegraphics[width=0.45\textwidth]{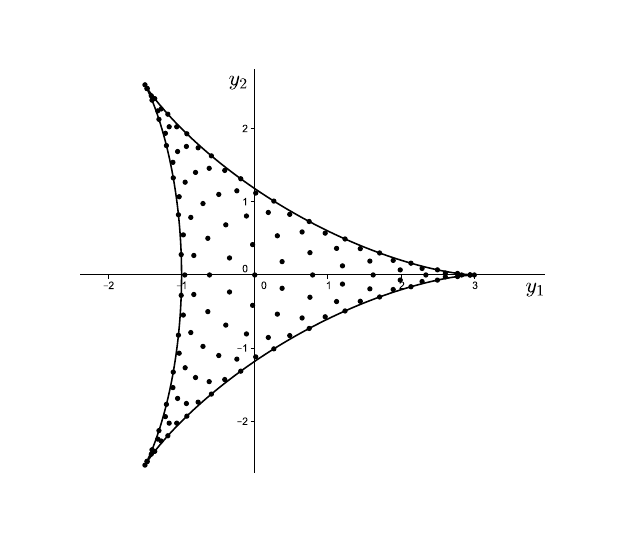}
\caption{The region $\Omega$ of $A_2$ together with the points of $\Omega_{15}$. The boundary of  $\Omega$ is defined by the equation $K(y_1,y_2)=0$.}
\label{A2} 
\end{figure}
Each point of $F_M$ as well as of $\Omega_M$ is labeled by the index set $ I_M$ and it is convenient for the point $x\in F_M$ and its image in $\Omega_M$ labeled by $j\in I_M$  to denote
$$ \ep_j =\ep(x) = \wt \ep (X_M (x)).  $$ 
The values of $\ep_j$ can be found in Table \ref{eps}. The cubature rule for any $p\in\Pi_{2M-1}$ is of the form
\begin{equation}\label{cuba2}
\int_{\Omega} p(y_1,y_2) K^{-\frac{1}{2}}(y_1,y_2)\,dy_1\,dy_2=\frac{\pi^2}{9M^2} \sum_{j\in I_M} \ep_j p(y_1^{(j)}, 	y_2^{(j)}).
\end{equation}

The cubature rule \eqref{cuba2} is an analogue of the formula deduced in \cite{XuA2} using the generalized cosine functions $TC_k$ and generalized sine  functions $TS_k$ defined by 
\begin{equation*}
\begin{aligned}
TC_k(t)&=\frac13\left[e^{\frac{i\pi}{3}(k_2-k_3)(t_2-t_3)}\cos{k_1\pi t_1}+e^{\frac{i\pi}{3}(k_2-k_3)(t_3-t_1)}\cos{k_1\pi t_2}+e^{\frac{i\pi}{3}(k_2-k_3)(t_1-t_2)}\cos{k_1\pi t_3}\right]\,,\\
TS_k(t)&=\frac13\left[e^{\frac{i\pi}{3}(k_2-k_3)(t_2-t_3)}\sin{k_1\pi t_1}+e^{\frac{i\pi}{3}(k_2-k_3)(t_3-t_1)}\sin{k_1\pi t_2}+e^{\frac{i\pi}{3}(k_2-k_3)(t_1-t_2)}\sin{k_1\pi t_3}\right]\,,
\end{aligned}
\end{equation*}
where $t=(t_1,t_2,t_3)\in \R^3$ with $t_1+t_2+t_3=0$ and $k=(k_1,k_2,k_3)\in\Z^3$ with $k_1+k_2+k_3=0$. 
It can be shown by performing the following change of variables and parameters
\begin{align*}&t_1=2a_1-a_2\,,\quad t_2=-a_1+2a_2\,,\quad t_3=-a_1-a_2\,,\\
&k_1=\lambda_1\,,\quad k_2=\lambda_2\,,\quad k_3=-\lambda_1-\lambda_2\,,\end{align*}
that generalized cosine and sine functions actually coincide (up to scalar multiplication) with the symmetric $C$-functions and antisymmetric $S$-functions of $A_2$. More precisely, we obtain
\begin{equation*}
TC_k(t)=\frac{h_\lambda}{6}C_\lambda(x)\,,\qquad TS_k(t)=\frac{1}{6}S_\lambda(t)\,.
\end{equation*}

\begin{example}\label{exa2}
The cubature formula \eqref{cuba2} is the exact equality of a weighted integral of any polynomial function of $m$-degree up to $M$ with a weighted sum of finite number of polynomial values. It can be used in numerical integration to approximate a weighted integral of any function by finite summing. 

If we choose the function $f(y_1,y_2)=K^{\frac12}(y_1,y_2)$ 
as our test function, then we can estimate the integral of $1$ over $\Omega$ 
$$\int_{\Omega}f(y_1,y_2)K^{-\frac12}(y_1,y_2) \,dy_1\,dy_2=\int_{\Omega}1 \,dy_1\,dy_2$$
by finite weighted sums with different $M$'s and compare the obtained results with the exact value of the integral of $1$ which is $2\pi\doteq 6.2832$. Table \ref{inta2} shows the values of the finite weighted sums for $M=10,20,30,50,100$.
\end{example}

\subsection{The case $C_2$}\

If the points are considered in the $\al^\vee-$basis, $x=a_1\alpha_1^\vee+a_2\alpha_2^\vee$, the symmetric  $C-$functions \eqref{orbit} and antisymmetric $S-$functions of $C_2$ are explicitly given by
\begin{equation*}
\begin{aligned}
C_{(\lambda_1,\lambda_2)}(a_1,a_2)&=\frac{2}{h_\lambda}\left(\cos{2\pi (\lambda_1 a_1+\lambda_2a_2)}+\cos{2\pi (-\lambda_1a_1+(\lambda_1+\lambda_2)a_2)}\right.\\ &\left.+\cos{2\pi ((\lambda_1+2\lambda_2)a_1-\lambda_2a_2)}+\cos{2\pi ((\lambda_1+2\lambda_2)a_1-(\lambda_1+\lambda_2)a_2)}\right)\,,\\
S_{(\lambda_1,\lambda_2)}(a_1,a_2)&=2\left(\cos{2\pi (\lambda_1a_1+\lambda_2a_2)}-\cos{2\pi (-\lambda_1a_1+(\lambda_1+\lambda_2)a_2)}\right.\\ &\left.-\cos{2\pi ((\lambda_1+2\lambda_2)a_1-\lambda_2a_2)}+\cos{2\pi ((\lambda_1+2\lambda_2)a_1-(\lambda_1+\lambda_2)a_2)}\right)\,.
\end{aligned}
\end{equation*} 
Performing the transform \eqref{var1}, the resulting real-valued functions $X_1,\,X_2$ are given by
\begin{equation*}
X_1=2(\cos{2\pi a_1}+\cos{2\pi(a_1-a_2)})\,,\quad X_2=2(\cos{2\pi a_2}+\cos{2\pi (2a_1-a_2)})\,.
\end{equation*}
The integral domain $\Omega$ can be described explicitly as
\begin{equation*}
\Omega=\set{(y_1,y_2)\in\R^2}{-2y_1-4\leq y_2\,,\; 2y_1-4\leq y_2\,,\; \frac14y_1^2\geq y_2}
\end{equation*}
The weight $K$-polynomial \eqref{KK} is given explicitly as
\begin{equation*}
K(y_1,y_2)=(y_1^2-4y_2)((y_2+4)^2-4y_1^2)\,.
\end{equation*}
The index set which will label the sets of points $F_M$ and $\Omega_M$ is introduced via
$$I_M=\set{[s_0,s_1,s_2]\in(\Z^{\geq0})^3}{s_0+2s_1+s_2=M}\,.$$
Thus the grid $F_M$ consists of the points
\begin{equation*}
F_M=\set{\frac{s_1}{M}\omega_1^\vee+\frac{s_2}{M}\omega_2^\vee}{[s_0,s_1,s_2]\in I_M}\,.
\end{equation*}
If for $j=[s_0,s_1,s_2]\in I_M$ we denote 
\begin{equation*}
(y_1^{(j)}, 	y_2^{(j)})= \left(X_1\left(\frac{2s_1+s_2}{2M},\frac{s_1+s_2}{M}\right),X_2\left(\frac{2s_1+s_2}{2M},\frac{s_1+s_2}{M}\right)\right)
\end{equation*}
then the set of nodes $\Omega_M$ consists of the points
\begin{equation*}
\Omega_M=\setb{(y_1^{(j)}, 	y_2^{(j)})\in \R^2}{ j\in I_M}\,.
\end{equation*}
The integration domain $\Omega$ is together with the set of nodes $\Omega_{15}$ depicted in Figure \ref{C2}. 
\begin{figure}
\includegraphics[width=0.45\textwidth]{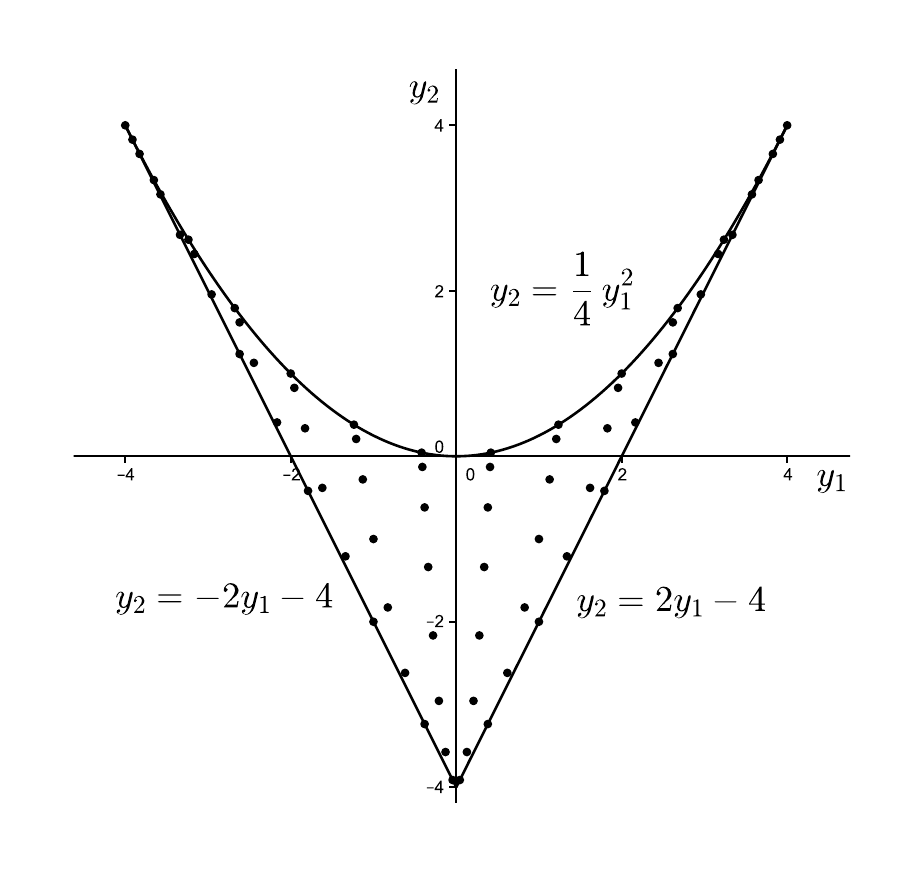}
\caption{The region $\Omega$ of $C_2$ together with with the points $\Omega_{15}$. The boundary is described  three equations $y_2=-2y_1-4$, $y_2=2y_1-4$ and $y_2=\frac14 y_1^2$.}
\label{C2} 
\end{figure}
Similarly to the case $A_2$, each point of $F_M$ as well as of $\Omega_M$ is labeled by the index set $I_M$ and it is convenient for the point $x\in F_M$ and its image in $\Omega_M$ labeled by $j\in I_M$ to denote
$$\varepsilon_j=\varepsilon(x)=\tilde\varepsilon(X_M(x))\,.$$
The values of $\varepsilon_j$ can be found in Table \ref{eps}.
\begin{table}
\begin{tabular}{c|ccc}
&&$\varepsilon_j$&\\\hline
$j\in I_M$&$A_2$&$C_2$&$G_2$\\
\hline\hline
$[\star,0,0]$&1&1&1\\
$[0,\star,0]$&1&2&3\\
$[0,0,\star]$&1&1&2\\
$[\star,\star,0]$&3&4&6\\
$[\star,0,\star]$&3&4&6\\
$[\star,\star,0]$&3&4&6\\
$[\star,\star,\star]$&6&8&12
\end{tabular}
\medskip
\caption{The values of $\varepsilon_j$ are shown for $A_2,C_2$ and $G_2$ with $\star$ denoting the corresponding coordinate different from $0$.}
\label{eps}
\end{table}
The cubature rule for any $p\in\Pi_{2M-1}$ takes the form
\begin{equation}\label{cubc2}
\int_{\Omega} p(y_1,y_2)K^{-\frac12}(y_1,y_2)dy_1\,dy_2=\frac{\pi^2}{4M^2}\sum_{j\in I_M}\varepsilon_j p(y_1^{(j)},y_2^{(j)}))\,.
\end{equation}

\begin{example}
The cubature formula \eqref{cubc2} is the exact equality of a weighted integral of any polynomial function of $m$-degree up to $M$ with a weighted sum of finite number of polynomial values. It can be used in numerical integration to approximate a weighted integral of any function by finite summing. 

Similarly to Example \ref{exa2}, if we choose the function $f(y_1,y_2)=K^{\frac12}(y_1,y_2)$ 
as our test function, then we can estimate the integral of $1$ over $\Omega$ 
$$\int_{\Omega}f(y_1,y_2)K^{-\frac12}(y_1,y_2) \,dy_1\,dy_2=\int_{\Omega}1 \,dy_1\,dy_2$$
by finite weighted sums with different $M$'s and compare the obtained results with the exact value of the integral of $1$ which is $\frac{32}{3}=10,666\bar{6}$. Table \ref{inta2} shows the values of the finite weighted sums for $M=10,20,30,50,100$. 
\end{example}
\subsection{The case $G_2$}\

If the points are considered in the $\al^\vee-$basis, $x=a_1\alpha_1^\vee+a_2\alpha_2^\vee$, the symmetric  $C-$functions \eqref{orbit} and antisymmetric $S-$functions of $G_2$ are explicitly given by
\begin{equation*}
\begin{aligned}
C_{(\lambda_1,\lambda_2)}(a_1,a_2)&=\frac{2}{|h_\lambda}\left(\cos{2\pi (\lambda_1a_1+\lambda_2a_2)}+\cos{2\pi (-\lambda_1a_1+(3\lambda_1+\lambda_2)a_2)}\right.\\ &\left.+\cos{2\pi ((\lambda_1+\lambda_2)a_1-\lambda_2a_2)}+\cos{2\pi ((2\lambda_1+\lambda_2)a_1-(3\lambda_1+\lambda_2)a_2)}
\right.\\ &\left.+\cos{2\pi ((-\lambda_1-\lambda_2)a_1+(3\lambda_1+2\lambda_2)a_2)}+\cos{2\pi ((-2\lambda_1-\lambda_2)a_1+(3\lambda_1+2\lambda_2)a_2)}\right)\,,\\
S_{(\lambda_1,\lambda_2)}(a_1,a_2)&=2\left(\cos{2\pi (\lambda_1a_1+\lambda_2a_2)}-\cos{2\pi (-\lambda_1a_1+(3\lambda_1+\lambda_2)a_2)}\right.\\ &\left.-\cos{2\pi ((\lambda_1+\lambda_2)a_1-\lambda_2a_2)}+\cos{2\pi ((2\lambda_1+\lambda_2)a_1-(3\lambda_1+\lambda_2)a_2)}
\right.\\ &\left.+\cos{2\pi ((-\lambda_1-\lambda_2)a_1+(3\lambda_1+2\lambda_2)a_2)}-\cos{2\pi ((-2\lambda_1-\lambda_2)a_1+(3\lambda_1+2\lambda_2)a_2)}\right)\,.
\end{aligned}
\end{equation*} 
Performing the transform \eqref{var1}, the resulting real-valued functions $X_1,\,X_2$ are given by
\begin{equation*}
\begin{aligned}
X_1&=2(\cos{2\pi a_1}+\cos{2\pi(a_1-3a_2)}+\cos{2\pi(2a_1-3a_2)})\,,\\ X_2&=2(\cos{2\pi a_2}+\cos{2\pi (a_1-a_2)}+\cos{2\pi(a_1-2a_2)})\,.
\end{aligned}
\end{equation*}
The integral domain $\Omega$ can be described explicitly as
\begin{equation*}
\Omega=\set{(y_1,y_2)\in\R^2}{-2((y_2+3)^{\frac32}+3y_2+6)\leq y_1\leq 2((y_2+3)^\frac32-3y_2-6),y_1\geq \frac14 y_2^2-3}\,.
\end{equation*}
The weight $K$-polynomial \eqref{KK} is given explicitly as
\begin{equation*}
K(y_1,y_2)=(y_2^2-4y_1-12)(y_1^2-4y_2^3+12y_1y_2+24y_1+36y_2+36)\,.
\end{equation*}
The index set which will label the sets of points $F_M$ and $\Omega_M$ is introduced via
$$I_M=\set{[s_0,s_1,s_2]\in(\Z^{\geq0})^3}{s_0+2s_1+3s_2=M}\,.$$
Thus the grid $F_M$ consists of the points
\begin{equation*}
F_M=\set{\frac{s_1}{M}\omega_1^\vee+\frac{s_1}{M}\omega_2^\vee}{[s_0,s_1,s_2]\in I_M}
\end{equation*}
If for $j=[s_0,s_1,s_2]\in I_M$ we denote 
\begin{equation*}
(y_1^{(j)}, 	y_2^{(j)})= \left(X_1\left(\frac{2s_1+3s_2}{M},\frac{s_1+2s_2}{M}\right),X_2\left(\frac{2s_1+3s_2}{M},\frac{s_1+2s_2}{M}\right)\right)
\end{equation*}
then the set of nodes $\Omega_M$ consists of the points
\begin{equation*}
\Omega_M=\setb{(y_1^{(j)}, 	y_2^{(j)})\in \R^2}{ j\in I_M}\,.
\end{equation*}
The integration domain $\Omega$ is together with the set of nodes $\Omega_{15}$ depicted in Figure \ref{G2}.
\begin{figure}
\includegraphics[width=0.45\textwidth]{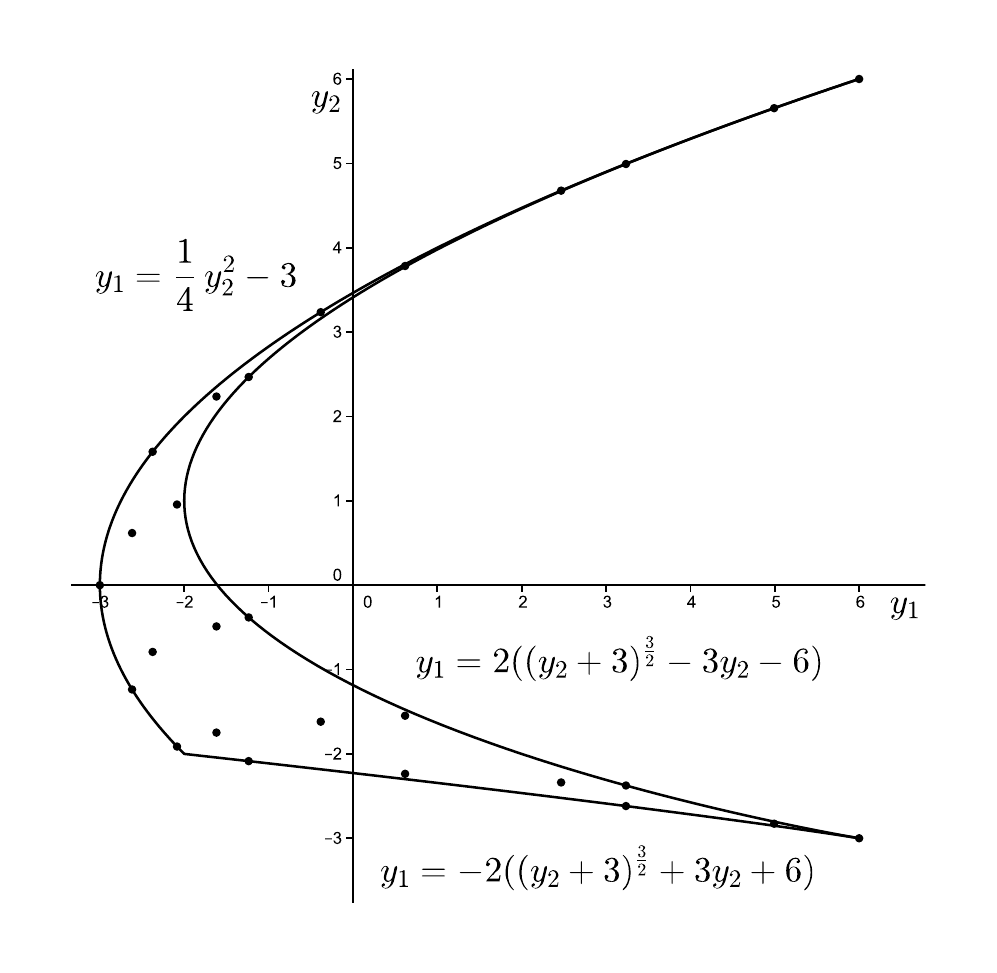}
\caption{The region $\Omega$ of $G_2$ together with the points of $\Omega_{15}$. The boundary of $\Omega$ is described  by three equations $y_1=\frac14 y_2^2-3$, $y_1=2((y_2+3)^\frac32-3y_2-6)$ and $y_1=-2((y_2+3)^\frac32+3y_2+6)$.}
\label{G2} 
\end{figure}
Similarly to the case $A_2$, each point of $F_M$ as well as of $\Omega_M$ is labeled by the index set $I_M$ and it is convenient for the point $x\in F_M$ and its image in $\Omega_M$ labeled by $j\in I_M$ to denote
$$\varepsilon_j=\varepsilon(x)=\tilde\varepsilon(X_M(x))\,.$$
The values of $\varepsilon_j$ can be found in Table \ref{eps}.
The cubature rule for any $p\in\Pi_{2M-1}$ is of the form 
\begin{equation}\label{cubg2}
\int_{\Omega} p(y_1,y_2)K^{-\frac12}(y_1,y_2)dy_1\,dy_2=\frac{\pi^2}{3M^2}\sum_{j\in I_M}\varepsilon_j p(y_1^{(j)},y_2^{(j)})\,.
\end{equation}

As in the case $A_2$, the cubature rule \eqref{cubg2} is similar to the Gauss-Lobatto cubature formula derived in \cite{XuG2}, where Xu et al study four types of functions $CC_k(t),SC_k(t),CS_k(t)$ and $SS_k(t)$ closely related to the orbit functions over Weyl groups. Concretely, the functions $CC_k(t)$ and $SS_k(t)$ are given by  
\begin{equation*}
\begin{aligned}
CC_k(t)&=\frac{1}{3}\left[\cos{\frac{\pi(k_1-k_3)(t_1-t_3)}{3}}\cos{\pi k_2 t_2}+\cos{\frac{\pi(k_1-k_3)(t_2-t_1)}{3}}\cos{\pi k_2 t_3}\right.\\&\left.+\cos{\frac{\pi(k_1-k_3)(t_3-t_2)}{3}}\cos{\pi k_2 t_1}\right]\,,\\
SS_k(t)&=\frac{1}{3}\left[\sin{\frac{\pi(k_1-k_3)(t_1-t_3)}{3}}\sin{\pi k_2 t_2}+\sin{\frac{\pi(k_1-k_3)(t_2-t_1)}{3}}\sin{\pi k_2 t_3}\right.\\&\left.+\sin{\frac{\pi(k_1-k_3)(t_3-t_2)}{3}}\sin{\pi k_2 t_1}\right]\,,
\end{aligned}
\end{equation*}
where the variable $t$ is given by homogeneous coordinates, i.e., $$t=(t_1,t_2,t_3)\in\R^3_H=\set{t\in \R^3}{t_1+t_2+t_3=0}$$ and parameter $k=(k_1,k_2,k_3)\in \Z^3\cap \R^3_H$. It can be verified that the linear transformations of variable and parameter
\begin{align*}&t_1=-a_1+3a_2\,,\quad t_2=2a_1-3a_2\,,\quad t_3=-a_1,\\
&k_1=\lambda_1+\lambda_2\,,\quad k_2=\lambda_1\,,\quad k_3=-2\lambda_1-\lambda_2\,,\label{par}\end{align*}
give the connection with $C-$functions and $S-$functions of  $G_2$: 
\begin{equation*}
C_\lambda(x)=\frac{12}{h_\lambda}CC_k(t)\,,\qquad S_\lambda(x)=-12SS_k(t)\,.\end{equation*} 

\begin{example}
The cubature formula \eqref{cubg2} is the exact equality of a weighted integral of any polynomial function of $m$-degree up to $M$ with a weighted sum of finite number of polynomial values. It can be used in numerical integration to approximate a weighted integral of any function by finite summing. 

Similarly to Example \ref{exa2}, if we choose the function $f(y_1,y_2)=K^{\frac12}(y_1,y_2)$ 
as our test function, then we can estimate the integral of $1$ over $\Omega$ 
$$\int_{\Omega}f(y_1,y_2)K^{-\frac12}(y_1,y_2) \,dy_1\,dy_2=\int_{\Omega}1 \,dy_1\,dy_2$$
by finite weighted sums with different $M$'s and compare the obtained results with the exact value of the integral of $1$ which is $\frac{128}{15}=8.533\bar{3}$. Table \ref{inta2} shows the values of the finite weighted sums for $M=10,20,30,50,100$. 

\begin{table}
\begin{tabular}{c|c|c|c|c|c}
$M$&$10$&$20$&$30$&$50$&$100$\\ \hline\hline
$A_2$&$6,0751$&$6,2314$&$6,2602$&$6,2749$&$6,2811$\\\hline
$C_2$&$10,056$&$10,5133$&$10,5985$&$10,6421$&$10,6605$\\\hline
$G_2$&$7,4789$&$8,2561$&$8,4092$&$8,4885$&$8,5221$
\end{tabular}
\bigskip
\caption{It shows the estimations of the integrals of $1$ over the regions $\Omega$ of $A_2,C_2$ and $G_2$ resp. by finite weighted sums of the right-hand side of \eqref{cuba2},\eqref{cubc2} and \eqref{cubg2} resp. for $M=10,20,30,50,100$.}
\label{inta2}
\end{table}
\end{example}

\section{Polynomial approximations}\label{secpol}
\subsection{The optimal polynomial approximation}\

Since the polynomial function \eqref{KK} is continuous and strictly positive in $\Omega^\circ$, its square root $K^{-\frac{1}{2}}$ can serve as a weight function for the weighted Hilbert space $\mathcal{L}_K^2(\Omega)$, i.e.
a space of complex-valued cosets of measurable functions $f$ such that $\int_{\Omega}|f|^2K^{-\frac{1}{2}}<\infty$ with an inner product defined by
\begin{equation}\label{sca}
(f,g)_K=\frac{1}{\kappa (2\pi)^n}\int_{\Omega}f(y)\overline{g(y)}K^{-\frac{1}{2}}(y)\,dy.
\end{equation}
Our aim is to consctruct a suitable Hilbert basis of $\mathcal{L}_K^2(\Omega)$.  Taking the inverse transforms  of \eqref{var1}, \eqref{var2} and substituting them into the polynomials \eqref{Cpoly} we obtain the polynomials $p_\la \in\C [y_1,\dots,y_n]$ such that
\begin{align}\label{CpolyX}
C_\la &= 	p_\la (X_1,\dots, X_n).
\end{align}
Moreover, successively applying Proposition \ref{subslem} to perform the substitutions \eqref{var2} in $p_\la$ and taking into account \eqref{degwtp} we obtain that 
\begin{equation}
	\mathrm{deg}_m \, p_\la=|\la|_m.
\end{equation}
Calculating the scalar product \eqref{sca} for the $p_\la$ polynomials \eqref{CpolyX}, we obtain that the continuous orthogonality of the $C-$functions \eqref{intor} is inherited, i.e.
\begin{equation}
	(p_\la,p_{\la'})_K=h_\la^{-1}\delta_{\la,\la'}, \q \la, \la' \in P^+ .
\end{equation}

Assigning to any function $f\in \mathcal{L}_K^2(\Omega)$ a function $\wt f\in \mathcal{L}^2(F)$ by the relation $\wt f (x)= f(X(x))$ and taking into account the expansion 
\eqref{expansion}
we obtain for its expansion coefficients $c_\la$ that
\begin{equation}\label{coefa}c_\lambda=\frac{h_\la}{|W||F|}\int_F \wt{f}(x)\overline{C_\lambda(x)}\,dx=\frac{h_\la}{\kappa (2\pi)^n}\int_{\Omega}f(y)\overline{p_\la(y)}K^{-\frac{1}{2}}(y)\,dy.
\end{equation}
Therefore any $f\in\mathcal{L}^2_K(\Omega)$ can be expanded in terms of $p_\lambda$,
\begin{equation}\label{exppol}
f=\sum_{\lambda\in P^+}a_\lambda p_\lambda\,,\q a_\lambda=h_\la\,(f,p_\lambda)_K\,
\end{equation}
and the set of $C-$polynomials $p_\la,\, \la \in P^+$ is a Hilbert basis of $ \mathcal{L}^2_K(\Omega)$.

To construct a basis of the space of multivariate polynomials $\Pi_M$ suffices to note that equation \eqref{mon} guarantess each monomial $Z_1^{\la_1}Z_2^{\la_2}\dots Z_n^{\la_n}$  can be expanded in terms of  $C_\la$ with $\la \in P^+_M$; the same can be said about the transformed monomials $X_1^{\la_1}X_2^{\la_2}\dots X_n^{\la_n}$.
Thus by the same argument as above we obtain for any $p\in \Pi_M$ the expansion 
\begin{equation}\label{pp}
p=\sum_{\la\in  P^+_M}b_\la p_\la\,,\q b_\lambda=h_\la\,(p,p_\lambda)_K.
\end{equation}
Truncating the series \eqref{exppol} to the finite set $P_M^+$ we obtain a polynomial approximation $u_M[f]\in \Pi_M$ of the functions  $f\in  \mathcal{L}^2_K(\Omega)$,
\begin{equation}\label{exppoltrunc}
u_{M}[f]=\sum_{\lambda\in P^+_M}a_\lambda p_\lambda\,,\q a_\lambda=h_\la\,(f,p_\lambda)_K.
\end{equation}

Relative to the  $\mathcal{L}^2_K(\Omega)$ norm is this approximation indeed optimal among all polynomials from  $\Pi_M$ as  states the following proposition.
\begin{tvr}
For any $f\in\mathcal{L}^2_K(\Omega)$ is the $u_{M}[f ]$ polynomial \eqref{exppoltrunc}
the best approximation of $f$, relative to the $\mathcal{L}^2_K(\Omega)$-norm, by any polynomial from $\Pi_M$.
\end{tvr}
\begin{proof}
Consider any $p\in \Pi_M$ a polynomial of the form \eqref{pp}, any  $f\in\mathcal{L}^2_K(\Omega)$ expanded by \eqref{exppol} and $u_{M}[f ]$ an approximation polynomial \eqref{exppoltrunc}. Then we calculate that
\begin{align*}
(f-p,f-p)_K & =(f,f)_K-(f,p)_K-(p,f)_K+(p,p)_K \\ 
&= (f,f)_K-\sum_{\lambda\in P^+_M}  h_\la^{-1}a_\lambda\overline{b_\lambda}-\sum_{\lambda\in P^+_M} h_\la^{-1}b_\lambda\overline{a_\lambda}+\sum_{\lambda\in P^+_M} h_\la^{-1} |b_\lambda|^2 \\
&= (f-u_{M}[f ],f-u_{M}[f ])_K+\sum_{\lambda\in P^+_M} h_\la^{-1}|b_\lambda-a_\lambda|^2\geq (f-u_{M}[f ],f-u_{M}[f ])_K\,.
\end{align*}
\end{proof}

\subsection{The cubature polynomial approximation}\

Rather than the optimal polynomial approximation \eqref{exppoltrunc} one may consider for practical applications its weakened version. Such a weaker version is obtained by using  the cubature formula for an approximate calculation of  $(f,p_\lambda)_K$, i.e. we set
\begin{equation}\label{exppoltruncnew}
v_{M}[f]=\sum_{\lambda\in P^+_M}a_\lambda p_\lambda\,,\q a_\lambda=\,\frac{h_\la}{c|W|M^n}\sum_{y\in \Omega_M}\wt\ep(y)f(y) \overline{ p_\la(y)}.
\end{equation}
Since for $f\in \Pi_{M-1}$ and $p_\la$, $\la \in P_M^+$ it holds that $f\overline{ p_\la}\in \Pi_{2M-1} $and the cubature formula is thus valid, we obtain that the optimal approximation coincides with $v_{M}[f]$, 
\begin{equation}
	v_{M}[f]= u_{M}[f],\q f\in \Pi_{M-1}.
\end{equation}

\begin{example}
As a specific example of a continuous model function in the case $C_2$, we consider 
$$f(y_1,y_2)=e^{-(y_1^2 + (y_2 + 1.8)^2)/(2\times0.35^2)}$$
defined on $\Omega$. 
The graph of $f$ together with its approximations $v_M[f]$ for $M=10,20,30$ is shown in Figure \ref{approx}. 
\begin{figure}
\includegraphics[scale=0.20]{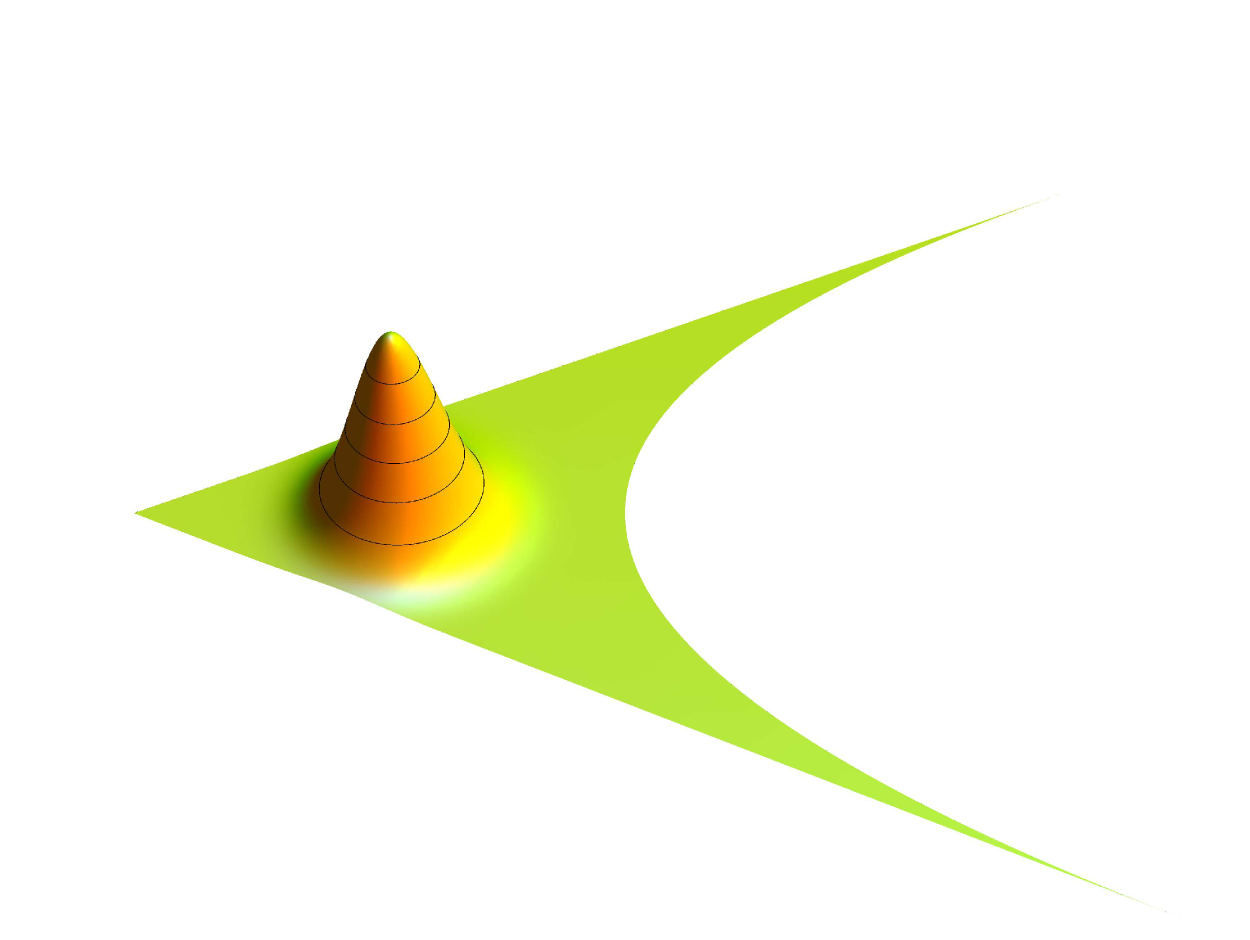}
\includegraphics[scale=0.20]{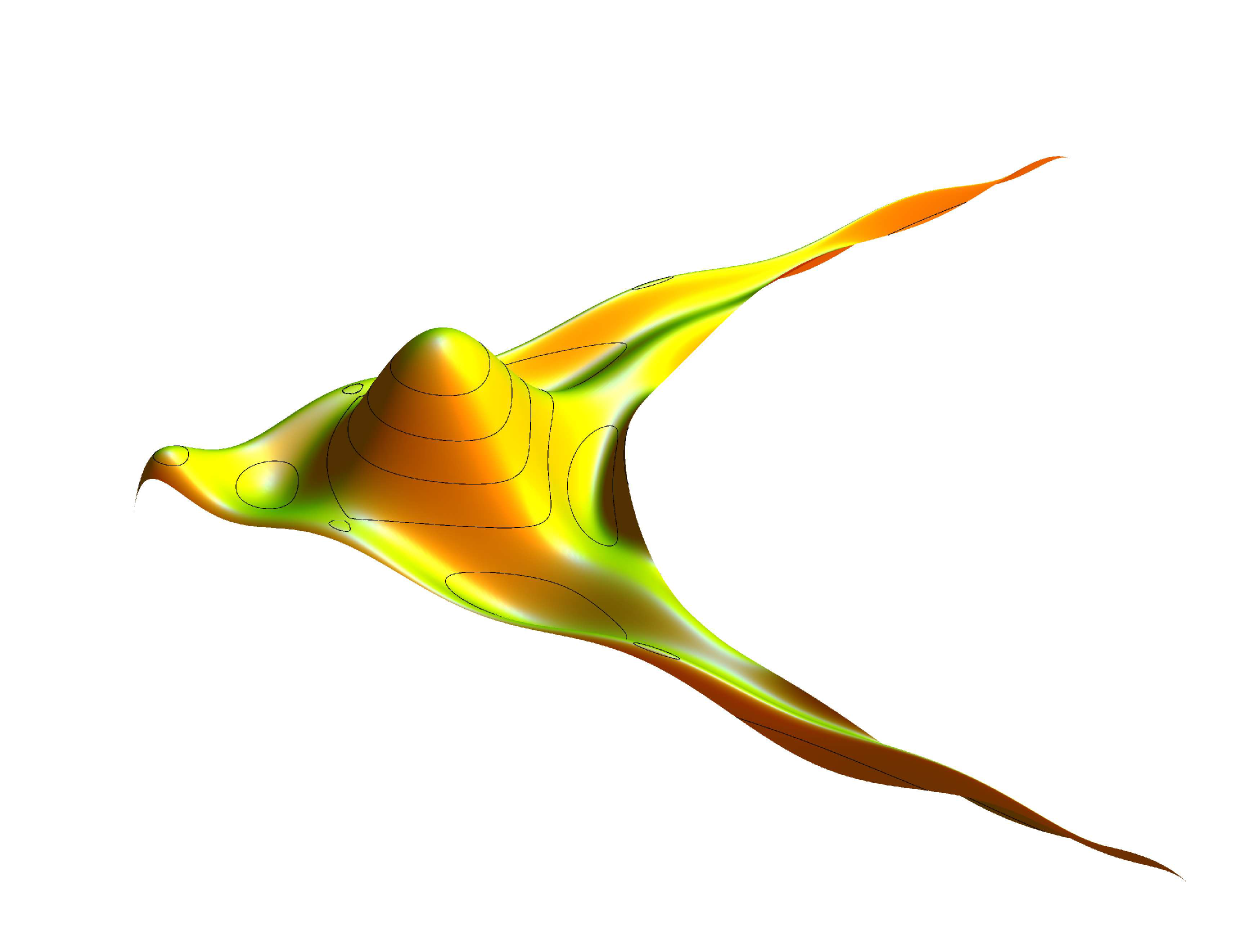}\\
\includegraphics[scale=0.20]{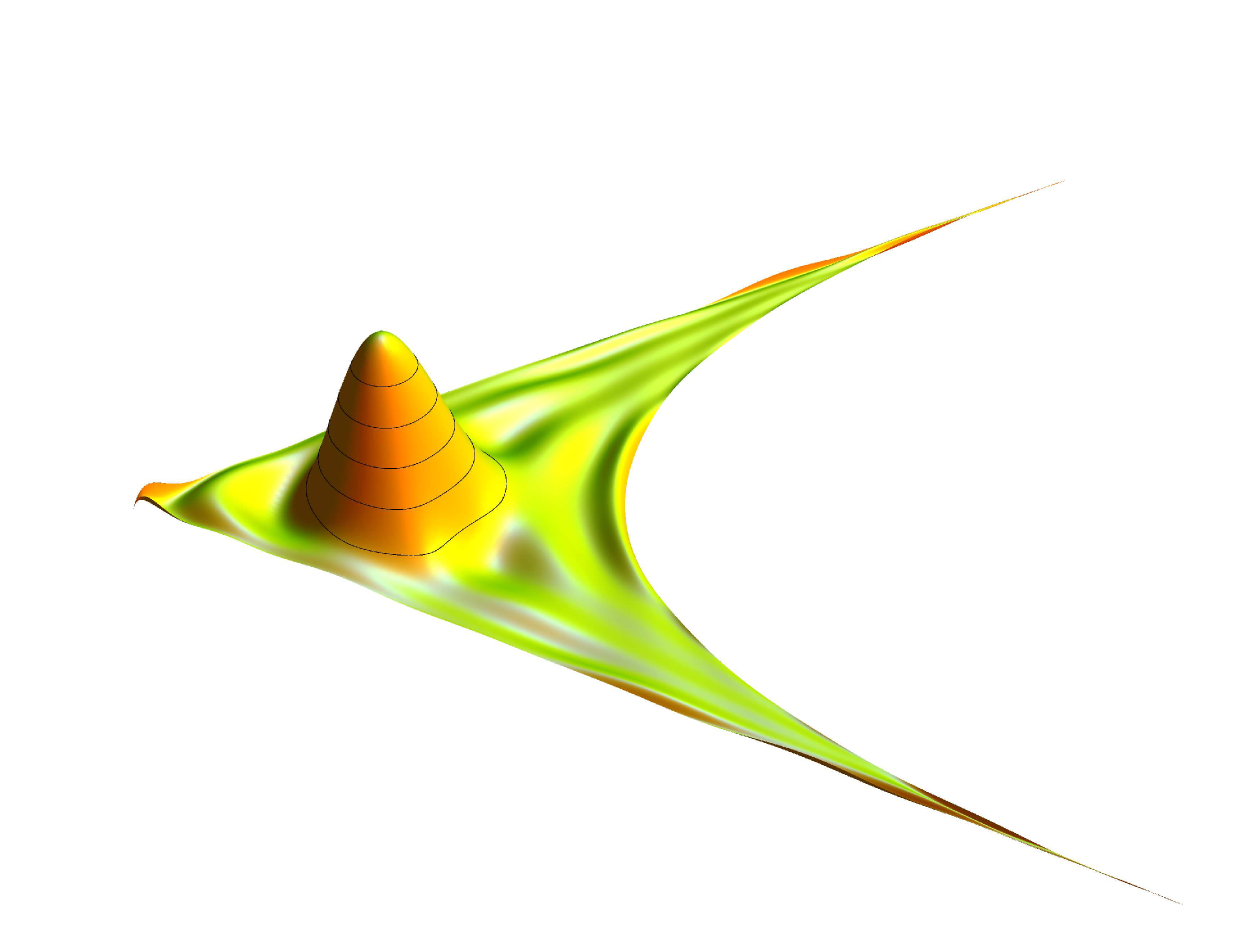}
\includegraphics[scale=0.20]{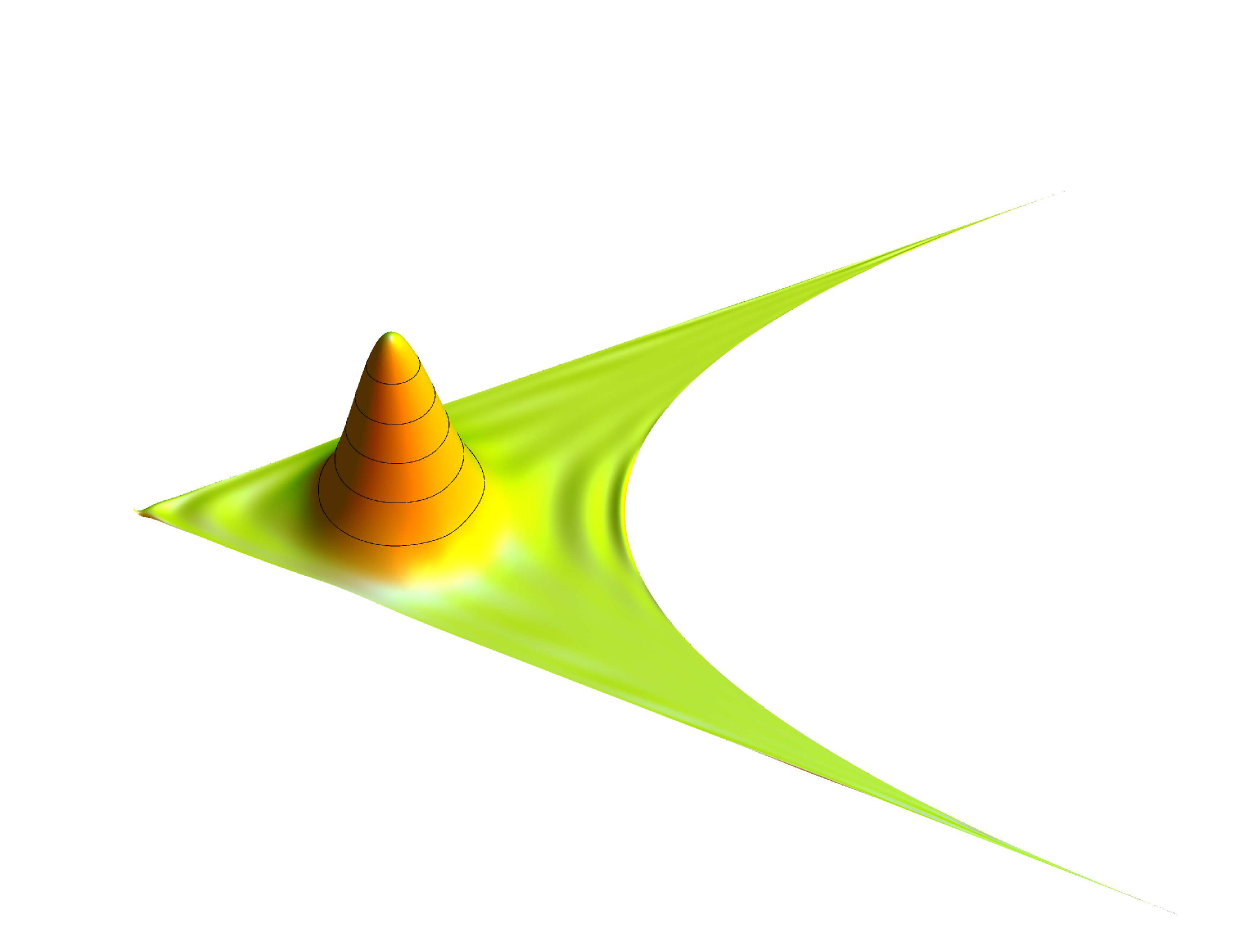}
\caption{The figure shows the model function $f$ and its approximations $v_M[f]$ for $M=10,20,30$ on $\Omega$ of $C_2$.}
\label{approx}
\end{figure}
Integral error estimates of the polynomial approximations $v_M[f]$
\begin{equation*}\int_{\Omega}|f(y_1,y_2)-v_M[f](y_1,y_2)|^2K^{-\frac12}(y_1,y_2)dy_1\,dy_2\end{equation*}
 can be found in Table \ref{errors}.
\begin{table}[h]
\begin{tabular}{c||c|c|c}
$M$&$10$&$20$&$30$\\ \hline
$\int_{\Omega}|f-v_M[f]|^2K^{-\frac12}dy_1\,dy_2$&$0,0636842$&$0,0035217$&$0,0000636$
\end{tabular}
\bigskip
\caption{The table shows the values of integral error estimates of the polynomial approximations $v_M[f]$ for $M=10,20,30$.}
\label{errors}
\end{table}
\end{example}

\section{Concluding remarks}

\begin{itemize}

\item Due to the generality of the present construction of the cubature formulas, some of the cases presented in this paper appeared already in the literature. The case $A_2$ is closely related to two-variable analogues of Jacobi polynomials on Steiner's hypocycloid \cite{koorn}; the case $C_2$ is related to two-variable analogues of Jacobi polynomials on a domain bounded by two lines and a parabola \cite{koorn,koorn1,koorn2} and the corresponding Gaussian cubature formulas induced by these polynomials are studied for example in \cite{XuC2}. The case $G_2$ and its cubature formulas are detailed in \cite{XuG2}.

\item The Chebyshev polynomials of the first kind induce the cubature formula of the maximal efficiency --- Gauss-Chebyshev quadrature \cite{Hand, Riv}. The nodes of this formula are $M$ roots of the Chebyshev polynomials of the first kind of degree $M$ and it exactly evaluates a weighted integral for any polynomial of degree at most $2M-1$. The set of nodes $\Omega_M$ of $A_1$ does not correspond to the set of roots of the Gauss-Chebyshev quadrature --- $\Omega_M$ of $A_1$ consists of $M+1$ points and includes two boundary points of the interval. Thus, the number of points $\Omega_M$ exceeds by one the minimum number of nodes and the resulting cubature formula is not optimal. This phenomenon, already observed for the $A_n$ sequence in \cite{LX}, generalizes to all cases of Weyl groups. Even though the developed formulas are slightly less efficient than the optimal ones, they contain points on the entire boundary of $\Omega_M$ --- this might be useful for some applications.

\item The technique for construction of the cubature formulas in this article is based on discrete orthogonality of $C-$functions over the set $F_M$. For the case $A_1$ the cosine transform corresponding to the discrete orthogonality of cosines is standardly known as discrete cosine transform of the type DCT-I \cite{Brit}. The $M$ roots of the Chebyshev polynomials of the first kind, which lead to the optimal cubature formula, enter the discrete orthogonality relations of the type DCT-II. This indicates that the discrete orthogonality relations of $C$-functions, which generalize the transforms of the type DCT-II, DCT-III and DCT-IV for admissible cases of Weyl groups \cite{CH},  might lead to cubature formulas of higher efficiency.

\item The weaker version of the polynomial approximation \eqref{exppoltrunc} assigns to any complex function $f: \Omega\map\C $ a polynomial functional series $\{
u_{M}[f]\}_{M=1}^\infty$.  
Existence of conditions for convergence of these functional series together with an estimate of the approximation error $\int \abs{f -u_{M}[f]  }^2K^{-\frac12}$ poses an open problem.

\item The Clenshaw-Curtis method for deriving cubature formulas is based on expressing a given function into a series of the corresponding set of orthogonal polynomials and then integrating the series term by term \cite{clenshaw,sloan}. The coefficients in the series are calculated using discrete orthogonality properties. The comparison of the resulting cubature formulas by the Clenshaw-Curtis technique and the method used in the present paper and in \cite{MP4,SsSlcub} deserves further study.

\item The polynomials of the orbit functions used in the present paper and in \cite{MP4,SsSlcub} to derive the cubature formulas are directly related to the Jacobi polynomials associated to root systems as well as to the Macdonald polynomials. The discrete orthogonality of the Macdonald polynomials with unitary parameters, achieved in \cite{DE}, opens a possibility of an extension of the current cubature formulas to the corresponding subset of the Macdonald polynomials. 

\end{itemize}

\section*{Acknowledgments}

This work is supported by the European Union under the project Support of
inter-sectoral mobility and quality enhancement of research teams at Czech Technical University in Prague CZ.1.07/2.3.00/30.0034, by the Natural Sciences and Engineering Research Council of Canada and by the Doppler Institute of the Czech Technical University in Prague. JH is also grateful for the hospitality extended to him at the Centre de recherches math\'ematiques, Universit\'e de Montr\'eal and gratefully acknowledge support by RVO68407700. LM would like to express her gratitude to the Department of Mathematics and Statistic at Universit\'e de Montr\'eal and the Institute de Sciences Math\'ematiques de Montr\'eal.

\end{document}